\documentclass[a4paper,12pt]{article}
\usepackage{amsmath}
\usepackage{amsfonts}
\usepackage{supertabular}
\usepackage[latin9]{inputenc}
\usepackage[english]{varioref}
\usepackage{dcolumn}
\usepackage[height=22cm , width = 16cm , top = 4cm , left = 3cm, a4paper]{geometry}
\usepackage{amsthm}
\usepackage{dsfont}
\usepackage{hyperref}
\usepackage[hypcap]{caption}
\usepackage{color}
\numberwithin{equation}{section}

\theoremstyle{plain}
\newtheorem{thm}{Theorem}[section]

\newtheorem{lem}[thm]{Lemma}
\newtheorem{prop}[thm]{Proposition}

\newtheorem{defn}[thm]{Definition}

\newcommand{\enter}{\bigskip}

\begin{document}
\author{Prasanta Kumar Barik \\
\footnotesize  Tata Institute of Fundamental Research Centre for Applicable Mathematics,\\  \small{Bangalore-560065, Karnataka, India}\\
  }
%\footnote{{\it{${}$ Email address: }} prasant.daonly01@gmail.com}
\title {Gelation in coagulation and multiple fragmentation equation with a class of singular rates }
\date{}
\maketitle

%%%%%%%%%%%%%%%%%%%%%%%%%%%% %%%%%%%%%%%%%%%%%
\hrule \vskip 8pt

%\textcolor{blue}{
\begin{quote}
{\small {\em\bf Abstract.} In this paper, a partial integro-differential equation modeling of coagulation and multiple fragmentation events is studied.
Our purpose is to investigate the global existence of gelling weak solutions to the continuous coagulation and multiple fragmentation
equation for a certain class of coagulation rate, linear selection rate and breakage function. Here, the coagulation rate has singularity for small mass
(size) and growing as polynomial function of mass for large particles whereas the breakage function attains singularity near the origin.
Moreover, a weak fragmentation process is considered for large mass particles to prove this result. The gelation transition is also discussed separately for both Smoluchowski coagulation equation and combined form of coagulation and multiple fragmentation equation. Finally, the long time behavior of solutions of the coagulation and multiple fragmentation is demonstrated.}
 \enter
\end{quote}

\hspace{.5cm}

\vskip 10pt \hrule

%%%%%%%%%%%%%%%%%%%%%%%%%%%%%%%%%%%%%%%%%%%%%%%%%%%%%%%%%%%%%%%%%%%
%%%%%%%%%%%%%%%%%%%%%%%%%%%%%%%%%%%%%%%%%%%%%%%%%%%%%%%%%%%%%%%%%%%
\section{Introduction}\label{Introduction1}
%%%%%%%%%%%%%%%%%%%%%%%%%%%%%%%%%%%%%%%%%%%%%%%%%%%%%%%%%%%%%%%%%%%
%%%%%%%%%%%%%%%%%%%%%%%%%%%%%%%%%%%%%%%%%%%%%%%%%%%%%%%%%%%%%%%%%%%
This paper concerns with the continuous coagulation and multiple fragmentation model which is a partial integro-differential equation. This model describes the dynamics of particle growth or decay under the assumption that each particle is identified by its mass (or size) which is denoted by a positive real number. The system under consideration can be considered as consisting of a large number of particles that can aggregate, by means of binary reactions, to form bigger particles or splits into smaller ones due to the reaction between particles and the wall in the system or spontaneously. This model has huge applications in different fields such as formation of stars and planets in astrophysics, raindrop breakage in meteorology science, polymerization in chemistry, aggregation of blood cells and cell division in biology and many more.\\

 Let $g$ be the distribution function of particles, then the dynamics of $g$ reads as \cite{Melzak:1957, McLaughlin:1997, Giri:2012, Giri:2013, Camejo:2015, Laurencot:2018, Barik:2019Existence}
\begin{align}\label{Cmfe}
\frac{\partial g (m, t) }{\partial t} = &
  \frac{1}{2} \int_{0}^{m} C^K(m-m^{\ast}, m^{\ast})g(m-m^{\ast}, t)g(m^{\ast}, t)dm^{\ast} \nonumber\\
  &+ \int_m^{\infty}b(m|m^{\ast})\mathcal{S}^R(m^{\ast})g(m^{\ast}, t)dm^{\ast}\nonumber\\
 & - \int_{0}^{\infty} C^K(m, m^{\ast})g(m, t)g(m^{\ast}, t)dm^{\ast} -\mathcal{S}^R(m)g(m, t),
\end{align}
with the initial datum
\begin{align}\label{Initialdata}
g(m, 0) = g^{in}(m)\ge 0~ \mbox{a.e.}
\end{align}

Here, the distribution function $g$ is a function of mass and time, and are represented by $m \in \mathds{R}_{>0} := (0, \infty)$ and $t \ge 0$, respectively. The familiar coagulation kernel $C^K(m, m^{\ast})$ is a non-negative and symmetric function which shows the rate at which particles of mass $m$ join with particles of mass $m^{\ast}$ to form a bigger one $m+m^{\ast}$ and  $b(m|m^{\ast})$ is the fragment distribution function describing the expected number of particles of mass $m$ produced given that a particle of mass $m^{\ast}$ undergoes a breakup event. The linear fragment rate is denoted by $\mathcal{S}^R(m)$ which is also known as the \emph{selection rate} and describes the rate at which the particle of mass $m$ is selected to break. Finally, the breakage function is assumed to satisfy the following properties
\begin{equation}\label{TNP}
\left.
\begin{split}
\int_0^{m^{\ast}} b(m|m^{\ast})dm =& \eta(m^{\ast})\ \forall m^{\ast} \in \mathds{R}_{>0},\ \text{where}\
\underset{m^{\ast}\in \mathds{R}_{>0}}{\sup} \eta(m^{\ast})=\eta <\infty,\\
 b(m|m^{\ast})=& 0 \ \ \ \forall \ m \ge m^{\ast},
\end{split}
\right\}
\end{equation}
and
\begin{align}\label{MCP}
\int_0^{m^{\ast}} m b(m|m^{\ast})dm =m^{\ast},\ \ \ \forall m\in (0, m^{\ast}).
\end{align}

In (\ref{TNP}), $\eta $ is the supremum of $\eta(m^{\ast})$ which is greater than equal to $2$. This $\eta(m^{\ast})$ denotes the total number of young particles obtained from the breakage of mother particles of mass $m^{\ast}$ and whereas (\ref{MCP}) shows that the total mass in the system remains conserved during the breakage of mother particle. Rest of the paper, we assume $b(m|m^{\ast})=(\gamma+2)\frac{m^{\gamma}}{{m^{\ast}}^{1+\gamma}}$, for $-1< \gamma \le 0$. One can easily check that this breakage function satisfies \eqref{TNP} and \eqref{MCP}.\\

The first term and the second term in the right-hand side to \eqref{Cmfe} give the birth of particles of mass $m$ after coalescing with particles of masses $m^{\ast}$ and, $m-m^{\ast}$ due to the coagulation process and the breakage of particle of mass $m^{\ast}$ into daughter particles due to the multiple fragmentation process, respectively. While the third and fourth terms describe the death of particles of mass $m$ due to both coagulation and multiple fragmentation events, respectively.\\

Now, we define the total mass of particles in the system for the distribution function $g$ associated to the coagulation and
multiple fragmentation equation (CMFE) as
 \begin{align}\label{Totalmass}
\mathcal{N}_1(t)=\mathcal{N}_1(g(t)):=\int_0^{\infty} m g(m, t)dm,  \ \ t \ge 0.
\end{align}

 From the conservation of matter we know that the mass is neither created nor destroyed but it changes from one state to another
state by some physical and chemical processes. Thus, in a closed system of particles, we expect that the total mass will also conserve
during both the coagulation and multiple fragmentation events. However, if the coagulation reaction is very high compared to the fragmentation
reaction, the large mass particles are merged very fast to form a giant particle in the system. Furthermore, this giant particle may go
away from the system. This process is known as \emph{gelation transition} and the finite time at which this process starts is known as the
 \emph{gelling time or gelation time} \cite{Escobedo:2003, Leyvraz:1981}.\\

The pure coagulation equation was first deduced by the pioneering work of Smoluchowski \cite{Smoluchowski:1917} for discrete case as a model for Brownian motion of particles which is also known as the \emph{Smoluchowski coagulation equation} (SCE) and later M\"{u}ller \cite{Muller:1928} gave the continuous version of this kind of model. Next, a combined effect of binary coagulation and multiple fragmentation is investigated by Melzak \cite{Melzak:1957} for continuous case. The global well-posedness for the continuous SCE, coagulation and binary fragmentation equation (CFE) and CMFE has been extensively studied for non-singular kernels in \cite{Melzak:1957, Stewart:1989, Stewart:1990, DaCosta:1994, DaCosta:1995, Dubovskii:1996, Laurencot:2000On, Escobedo:2002, Laurencot:2002, Escobedo:2003, Giri:2012, Giri:2013, Laurencot:2015, Barik:2017Anote} and references therein. In addition, the existence of gelling and mass-conserving solutions have been investigated in \cite{Escobedo:2002, Escobedo:2003, Laurencot:2000On}. While the existence and uniqueness of solutions to the continuous CMFE with singular rates have been discussed in \cite{Barik:2018Mass, Camejo:2015, Laurencot:2018, Norris:1999, Laurencot:2018Uni, Laurencot:2019MassSelf, Laurencot:2019Mass2}. The existence of weak solutions to the continuous SCE locally in time is studied by Norris \cite{Norris:1999} when the coagulation rate $A$ that satisfies $A(m, m^{\ast}) \le a(m) a(m^{\ast})$, with $a: (0, \infty) \rightarrow [0, \infty)$ and $a(c m)\le c a(m)$  for all $m \in (0, \infty)$, $c \ge 1$, where $c$ is a sub-linear function and the initial data $g^{in} \in L^1((0, \infty); a(m)^2)$. By imposing an additional condition on the coagulation rate, i.e., for $\epsilon >0$ such that $\epsilon m \le a(m)$, the mass-conserving property of the solution is also studied. By using a weak $L^1$ compactness method, Camejo and Warnecke \cite{Camejo:2015} have studied the existence of solutions to the continuous CMFE for the singular rate, where the rate $A_2$ and the selection rate $S_1$, respectively, satisfy
 \begin{equation*}
 A_2(m, m^{\ast}) \le k (1+m)^{\alpha}(1+m^{\ast})^{\alpha}({m}{m^{\ast}})^{-\sigma},\ \text{for}\  \sigma \in [0, 1/2), \alpha -\sigma \in [0,1)\ \text{and}\ k>0,
 \end{equation*}
 and
 \begin{equation*}
S_1(y) \le k' m^{\beta}\ \text{where}\ \beta \in (0, 1)\ \text{and}\ k'>0.
 \end{equation*}
In addition, the uniqueness result is shown for $A_2$ when $\alpha =0$. Recently, we have studied the existence of mass-conserving solutions to the continuous SCE having linear growth for large masses and singularity for small mass particles for coagulation rate whatever the approximations to the original problems, see \cite{Barik:2018Mass} and also we have studied a similar type of result for the continuous CMFE where the selection rate satisfies a linear growth, see \cite{Barik:2019Existence}. In \cite{Laurencot:2000On}, the author has investigated the existence of weak solutions to a particular type of coagulation and multiple fragmentation equation for product type of coagulation rate with a weak fragmentation. In addition, the gelation transition is also discussed. The novelty of the present work is that we have considered a more general coagulation rate i.e., the coagulation rate has singularity for small size particles and polynomial growth for large one. However, to cover this growth, a weak fragmentation condition is considered on the selection rate. This work is a generalization of the previous work \cite{Laurencot:2000On} and \cite{Barik:2019Existence}. The motivation of the present work is form \cite{Laurencot:2000On}.\\

The rest of paper is organized as follows. In Section $2$, we summarize some assumptions on the initial data, coagulation kernel, selection rate and the breakage function used throughout the paper. In addition, some lemmas and main theorems of this paper is also stated in this section. Theorem $2.2$ is proved by using a weak $L^1$ compactness method in Section $3$. In Section $4$, gelation transition is studied for the pure coagulation equation. Finally, the gelation transition and asymptotic behaviour of the solutions is discussed for the CMFE in the last section.

%%%%%%%%%%%%%%%%%%%%%%%%%%%%%%%%%%%%%%
%%%%%%%%%%%%%%%%%%%%%%%%%%%%%%%%%%%%%%
\section{Main results}
%%%%%%%%%%%%%%%%%%%%%%%%%%%%%%%%%%%%%%
%%%%%%%%%%%%%%%%%%%%%%%%%%%%%%%%%%%%%%
Our main goal in this paper is to show the existence of gelling solutions to \eqref{Cmfe}--\eqref{Initialdata} when the coagulation rate $C^K$ and the breakage function $b$ have singularity near the origin and the selection rate $\mathcal{S}^R$ is unbounded for large size particles. Moreover, the coagulation rate $C^K$ has polynomial growth for large size particles. In addition, a weak fragmentation condition has imposed for the large size particles. More precisely, we make the following assumptions on the initial data $g^{in}$, the coagulation rate $C^K$, selection rate $\mathcal{S}^R$ and breakage function $b$:\\
Let $\sigma \in [0, (1+\gamma)/2)$ such that
\begin{equation}\label{coagulation kernel}
\left.
\begin{split}
0 \le C^K(m, m^{\ast})  & = C^K(m^{\ast}, m)  = k_1 (m m^{\ast} )^{-\sigma},\  (m, m^{\ast})\in (0, 1)^2 , \\
0 \le C^K(m, m^{\ast})  &=  C^K(m^{\ast}, m) = k_1 \Gamma(m^{\ast}) {m}^{-\sigma},\  (m, m^{\ast}) \in (0, 1) \times [1, \infty),\\
0 \le C^K(m, m^{\ast})  &=  C^K(m^{\ast}, m)  = k_1 \Gamma(m) \Gamma(m^{\ast}),\  (m, m^{\ast}) \in [1, \infty)^2,
\end{split}
\right\}
\end{equation}
where $k_1>0$ and $\Gamma$ is a polynomial function.\\
 There exists a positive constant $k_2 = \frac{(\gamma +2)}{(1+\gamma-2\sigma)} >2$ such that
\begin{eqnarray}\label{Breakage function1}
\int_0^{m^{\ast}} m^{-2\sigma}b(m|m^{\ast})dm \le k_2 {m^{\ast}}^{-2\sigma}.
\end{eqnarray}
The selection rate $\mathcal{S}^R$ satisfies
\begin{align}\label{Selection Rate}
\mathcal{S}^R(m) \le k_3 \varphi(m) m^{1+\gamma},\ \forall m \in \mathds{R}_{>0},
\end{align}
where $k_3 \ge 0$, and a non-decreasing and non-negative function $\varphi$ satisfies
\begin{align}\label{Nondecreasing function}
\lim_{m \to \infty} \varphi(m) =0.
\end{align}
Furthermore,
\begin{align}\label{Selection RateAdditional}
\mathcal{S}^R(m) \in L^{\infty}(0, \lambda),
\end{align}
where $\lambda$ is a positive real constant. In addition, there exists a $p \in(1, 2)$ (depending $\sigma$ and $\gamma$) such that $p(\gamma-\sigma)+1>0$.\\ \\
Note that from \eqref{Selection Rate} and \eqref{Nondecreasing function}, it is clear that the rate of splitting of large size particles are very slow. Thus, the condition \eqref{Selection Rate} along with \eqref{Nondecreasing function} is also known as the \emph{weak fragmentation}.\\
Finally, the initial data $g^{in}$ enjoys the following
\begin{align}\label{Initialdatacondition}
 g^{in} \in L^1_{-2\sigma, 1}(\mathds{R}_{>0}),
 \end{align}
where $L^1_{-2\sigma, 1}(\mathds{R}_{>0})$ is a Banach space  and
\begin{align}\label{Initialdatafinite}
\mathcal{Q}:=\int_0^{\infty} (m+m^{-2\sigma}) g^{in}(m) dm < \infty.
\end{align}
Next, we state the definition of weak solutions to \eqref{Cmfe}--\eqref{Initialdata}.
\begin{defn}\label{definition weak}
 Let the initial data $g^{in} \ge 0$ a.e. and satisfies \eqref{Initialdatacondition}. Then a weak solution to \eqref{Cmfe}--\eqref{Initialdata} is a non-negative function $g \in \mathcal{C}([0,t]; L^1(\mathds{R}_{>0} ) ) \cap L^{\infty}(0, t; L_{-2\sigma, 1}^1( \mathds{R}_{>0}))$ such that
 \begin{align*}
& \int_0^{\infty} \int_0^{\infty} C^K(m, m^{\ast}) g(m, t) g(m^{\ast}, t) dm^{\ast} dm  \in L^1(0, t),\nonumber\\
& \int_0^{\infty} \int_m^{\infty} b(m|m^{\ast}) \mathcal{S}^R(m^{\ast}) g(m^{\ast}, t) dm^{\ast} dm  \in L^1(0, t),
 \end{align*}
  and
  \begin{align}\label{defweaksolidentity}
 g(m, t) = & g^{in}(m) + \int_0^t   \bigg\{ \frac{1}{2} \int_{0}^{m} C^K(m-m^{\ast}, m^{\ast})g(m-m^{\ast}, \tau)g(m^{\ast}, \tau)dm^{\ast} \nonumber\\
  &+ \int_m^{\infty}b(m|m^{\ast})\mathcal{S}^R(m^{\ast})g(m^{\ast}, \tau)dm^{\ast}\nonumber\\
 & - \int_{0}^{\infty} C^K(m, m^{\ast})g(m, \tau)g(m^{\ast}, \tau)dm^{\ast} -\mathcal{S}^R(m)g(m, \tau) \bigg\} d\tau
 \end{align}
 for almost every $m \in \mathds{R}_{>0} $ and $t \in (0, T]$, where $T \in \mathds{R}_{>0}$.
\end{defn}
 Now our main results can be stated as follows:
\begin{thm}\label{Theorem1}
Assume that \eqref{coagulation kernel}--\eqref{Initialdatacondition} hold. Then there exists at least one solution (in the sense of definition \ref{definition weak}) to \eqref{Cmfe}--\eqref{Initialdata} on $[0, \infty)$ satisfying
\begin{equation}\label{Initialmass}
\mathcal{N}_1(t) \le \mathcal{N}_1^{in}:=\int_0^{\infty} m g^{in}(m) dm.
\end{equation}
\end{thm}

\begin{thm}\label{Theorem2}
Assume that \eqref{Initialdatacondition} holds and $\mathcal{S}^R \equiv 0$. Assume further that there exists $\lambda >1$ such that
\begin{align}\label{condicoag}
\Gamma(m) \ge \lambda m^{1+\sigma},\ \ \ m \ge 0,
\end{align}
\begin{equation}\label{coagkerGel}
 C^K(m, m^{\ast})  = k_1 \frac{\Gamma(m) \Gamma(m^{\ast})}{(m m^{\ast})^{\sigma}},\  (m, m^{\ast}) \in (0,, \infty)^2.
\end{equation}
 Let $g$ be a solution to \eqref{Cmfe}--\eqref{Initialdata} on $[0, \infty)$, then
\begin{equation}\label{GelationCoa1}
\mathcal{N}_1(t) \le \frac{ \sqrt{2 \mathcal{N}_0^{in}}    }{\lambda \sqrt{k_1}} t^{-\frac{1}{2}}.
\end{equation}
If
\begin{align}\label{GelationCoa2}
\mathcal{I}_{p}(t) := \int_0^{\infty} m^{-p} g^{in}(m) dm < \infty,
\end{align}
for some $p \in (0, \infty)$, then
\begin{align}\label{GelationCoa3}
\mathcal{N}_1(t) \le \mathcal{N}_1^{in} \bigg\{ 1+t T^{\dag} \bigg\}^{\frac{-(p+1)}{(p+2)}},
\end{align}
where
\begin{align*}
T^{\dag} =\frac{(p+2)}{2p}  { \mathcal{N}_1^{in}}^{\frac{(p+2)}{(p+1)}} k_1 \lambda^2  \mathcal{I}_p^{\frac{-1}{(p+1)}}.
\end{align*}
Finally, if $g^{in} \equiv 0$ on $(0, \delta)$ for some $\delta >0$, we have
\begin{align}\label{GelationCoa4}
\mathcal{N}_1(t)  \le  \mathcal{N}_1^{in}\sqrt{2} \bigg[ 2+ k_1 \delta \lambda^2 t {\mathcal{N}_1^{in}}^2  \bigg]^{-\frac{1}{2}}.
\end{align}
\end{thm}
\begin{thm}\label{Theorem3}
Assume that \eqref{Breakage function1} and \eqref{Initialdatacondition}--\eqref{Initialdatafinite} hold. Assume further that \eqref{condicoag} and \eqref{coagkerGel} hold. Furthermore, the selection rate satisfies 
\begin{align}\label{SelectionRate1}
\mathcal{S}^R(m) \le k_3 \varphi(m) m,
\end{align}
where $\varphi$ is defined in \eqref{Nondecreasing function}. Let $g$ be a solution to \eqref{Cmfe}--\eqref{Initialdata} on $[0, \infty)$, then
 \begin{align}\label{GelationCMFEeq2}
\mathcal{N}_1(t)  \le \frac{k_3}{\lambda^2}  (\eta -1) \varphi(0) +& \bigg[ \frac{k_3^2}{\lambda^4} (\eta-1)^2 \varphi(0)^2+\frac{2\mathcal{Q}} {t\lambda^2}  \bigg]^{1/2}
\end{align}
Then gelation occurs in a finite time. Moreover, as $t \to \infty$, then
\begin{align}\label{GelationCMFEeq3}
\lim_{t \to \infty}\mathcal{N}_1(t)  \le \frac{2k_3}{\lambda^2}  (\eta -1) \varphi(0).
\end{align}
\end{thm}

%%%%%%%%%%%%%%%%%%%%%%%%%%%%%%%%%%%%%%%%
%%%%%%%%%%%%%%%%%%%%%%%%%%%%%%%%%%%%%%%%
\section{Existence of solutions}
%%%%%%%%%%%%%%%%%%%%%%%%%%%%%%%%%%%%%%%%
%%%%%%%%%%%%%%%%%%%%%%%%%%%%%%%%%%%%%%%%
Before turning to the proof of the main existence result in Theorem \ref{Theorem1} to \eqref{Cmfe}--\eqref{Initialdata}, we first discuss the existence and uniqueness of solutions to truncated problems \eqref{Cmfe}--\eqref{Initialdata}. For this, we truncate the coagulation rate and the selection rate by using a suitable compact support $(1/n,  n)$. For each $n \in \mathds{N}$, we define a sequence of approximations of $C^K$, $\mathcal{S}^R$ and $g^{in}$, respectively,
\begin{align}\label{truncatedcoaker}
C^K_n(m, m^{\ast})= C^K(m, m^{\ast}) \chi_{(1/n, n)}(m) \chi_{(1/n, n)}(m^{\ast}),
\end{align}
\begin{align}\label{truncatedselrate}
\mathcal{S}^R_n(m)=\mathcal{ S}^R(m) \chi_{(0, n)}(m),
\end{align}
and
\begin{align}\label{truncatedinidata}
g^{in}_n(m)= g^{in}(m) \chi_{(0, n)}(m).
\end{align}
From \eqref{truncatedcoaker} and \eqref{truncatedselrate}, it is clear that both the coagulation kernel and the selection rate are
bounded for each $n$. Hence, $\Gamma_n(m)=\Gamma(m) \chi_{(0, n)}(m)$ is also bounded for each $n$.\\
Now, setting \eqref{truncatedcoaker}--\eqref{truncatedinidata} into \eqref{Cmfe}--\eqref{Initialdata}, we obtain
\begin{align}\label{Cmfetruncated}
\frac{\partial g_n (m, t) }{\partial t}  = & \frac{1}{2} \int_{0}^{m} C^K_n(m-m^{\ast}, m^{\ast})g_n(m-m^{\ast}, t)g_n(m^{\ast}, t)dm^{\ast} \nonumber\\
 &+ \int_m^{n}b(m|m^{\ast})\mathcal{S}^R_n(m^{\ast})g_n(m^{\ast}, t)dm^{\ast} \nonumber\\
 & - \int_{0}^{n} C^K_n(m, m^{\ast})g_n(m, t) g_n(m^{\ast}, t)dm^{\ast} - \mathcal{S}^R_n(m)g_n(m, t),
\end{align}
with initial value
\begin{align}\label{Initialdatatruncated}
g_n(m, 0) = g^{in}(m)\chi_{(0, n)}(m) \ge 0~ \mbox{a.e.}
\end{align}
We now turn to state the existence and uniqueness result to \eqref{Cmfetruncated}--\eqref{Initialdatatruncated}.
\begin{prop}\label{Prop1}
Let $T \in (0, \infty)$ and $n \ge 1$. Then for each $n$, there exists  a unique non-negative solution $g_n$ to \eqref{Cmfetruncated}--\eqref{Initialdatatruncated} such that $g_n \in \mathcal{C}^1([0, T], L^1(0, n))$ and for every $(m, t) \in (0, n) \times \mathds{R}_{>0}$, it satisfies
\begin{align}\label{Truncatedsolutions}
g_n(m, t)=g_n^{in}(m) & + \int_0^t \bigg\{  \frac{1}{2} \int_{0}^{m} C^K_n(m-m^{\ast}, m^{\ast})g_n(m-m^{\ast}, \tau)g_n(m^{\ast}, \tau)dm^{\ast} \nonumber\\
 &+ \int_m^{n}b(m|m^{\ast})\mathcal{S}^R_n(m^{\ast})g_n(m^{\ast}, \tau)dm^{\ast} \nonumber\\
 & - \int_{0}^{n} C^K_n(m, m^{\ast})g_n(m, \tau) g_n(m^{\ast}, \tau)dm^{\ast} - \mathcal{S}^R_n(m)g_n(m, \tau) \bigg\} d\tau,
\end{align}
where $ t\in (0, T]$.
\end{prop}
%%%%%%%%%%%%%%%%%%%%%%%%%%%%%%%%%%%%%%%
%%%%%%%%%%%%%%%%%%%%%%%%%%%%%%%%%%%%%%%

%%%%%%%%%%%%%%%%%%%%%%%%%%%%%%%%%%%%%%%%%
%%%%%%%%%%%%%%%%%%%%%%%%%%%%%%%%%%%%%%%%%
To establish the existence result in our framework, the following identity is crucial:
\begin{lem}\label{Lemma1}
Let $\Theta$ be a locally bounded function on $\mathds{R}_{>0}$. Then for each $n \in \mathds{N}$, $ T \in (0, \infty )$ and $\lambda^{\ast} \in (0, n]$,
we have
\begin{align}\label{1truncatedidentity}
\int_0^{\lambda^{\ast}} \{ g_n(m, t) & - g_n^{in}(m) \} \Theta(m)dm \nonumber\\
=&\frac{1}{2}\int_0^t \int_0^{\lambda^{\ast}} \int_{0}^{\lambda^{\ast}} \tilde{\Theta}_{\lambda^{\ast}}(m, m^{\ast}) C^K_n(m, m^{\ast})g_n(m, \tau) g_n(m^{\ast}, \tau)dm^{\ast} dm d\tau \nonumber\\
&- \int_0^t \int_0^{\lambda^{\ast}} \int_{ \lambda^{\ast} }^{n} \Theta(m)  C^K_n(m, m^{\ast})g_n(m, \tau) g_n(m^{\ast}, \tau)dm^{\ast} dm d\tau \nonumber\\
&+\int_0^t \int_0^{\lambda^{\ast}} \Pi_{\Theta}(m)\mathcal{S}_n^R(m) g_n(m, \tau)dm d\tau \nonumber\\
&+\int_0^t \int_{ \lambda^{\ast} }^{n}  \int_0^{\lambda^{\ast}} \Theta(m) b(m|m^{\ast}) \mathcal{S}_n^R(m^{\ast}) g_n(m^{\ast}, \tau)dm dm^{\ast} d\tau,
\end{align}
where
\begin{align}\label{1Identity1}
\tilde{\Theta}_{\lambda^{\ast}}(m, m^{\ast}):=\Theta(m+m^{\ast})\chi_{(0, \lambda^{\ast})}(m+m^{\ast})-\Theta(m)-\Theta(m^{\ast})
\end{align}
and
\begin{align}\label{1Identity2}
\Pi_{\Theta}(m):=\int_0^m b(m^{\ast}|m)\Theta(m^{\ast})dm^{\ast}- \Theta(m).
\end{align}
In particular, if $\lambda^{\ast} =n$, we have the following identity
\begin{align}\label{truncatedidentity}
\int_0^{n} \{ g_n(m, t) & - g_n^{in}(m) \} \Theta(m)dm \nonumber\\
=&\frac{1}{2}\int_0^t \int_0^{n} \int_{0}^{n}\tilde{\Theta}_n(m, m^{\ast}) C^K_n(m, m^{\ast})g_n(m, \tau) g_n(m^{\ast}, \tau)dm^{\ast} dm d\tau \nonumber\\
&+\int_0^t \int_0^{n} \Pi_{\Theta}(m)\mathcal{S}_n^R(m) g_n(m, \tau)dm d\tau,
\end{align}
where
\begin{align}\label{Identity1}
\tilde{\Theta}_n (m, m^{\ast}):=\Theta(m+m^{\ast})\chi_{(0, n)}(m+m^{\ast})-\Theta(m)-\Theta(m^{\ast}).
\end{align}
 Furthermore, $g_n$ satisfies
\begin{align}\label{truncated mass}
\int_0^n m g_n(m, t) dm \le \int_0^n m g_n^{in}(m) dm,
\end{align}
for every $t \in [0, T]$.
\end{lem}
\begin{proof}
Let us multiply \eqref{Truncatedsolutions} by $ \Theta(m) $ and integrating from $0$ to $\lambda^{\ast}$ with respect to mass variable $m$, this gives
\begin{align}\label{IdentityE1}
\int_0^{\lambda^{\ast}} \Theta(m) & \{ g_n(m, t)-g_n^{in}(m) \} dm \nonumber\\
= & \frac{1}{2} \int_0^t \int_0^{\lambda^{\ast}} \int_0^m \Theta(m) C_n^K(m-m^{\ast}, m^{\ast}) g_n(m-m^{\ast}, \tau) g_n(m^{\ast}, \tau) dm^{\ast} dm d\tau \nonumber\\
&- \int_0^t \int_0^{\lambda^{\ast}}\int_0^n \Theta(m) C_n^K(m, m^{\ast}) g_n(m, \tau) g_n(m^{\ast}, \tau) dm^{\ast} dm d\tau \nonumber\\
&+ \int_0^t \int_0^{\lambda^{\ast}} \int_m^n \Theta(m) b(m|m^{\ast}) \mathcal{S}^R_n(m^{\ast}) g_n(m^{\ast}, \tau) dm^{\ast} dm d\tau \nonumber\\
&- \int_0^t \int_0^{\lambda^{\ast}}\Theta(m) \mathcal{S}_n^R(m) g_n(m, \tau) dm d\tau.
\end{align}
Now we use Fubini's theorem to the first and third terms on the right-hand side of \eqref{IdentityE1}. By applying the transformation $m-m^{\ast}=m'$ \& $m^{\ast}={m^{\ast}}'$ to the first term and the symmetry of $C_n^K$, and finally changing $m'$ to $m$ \& ${m^{\ast}}'$ to $m^{\ast}$, we get
\begin{align*}
\int_0^{\lambda^{\ast}} \Theta(m) & \{ g_n(m, t)- g_n^{in}(m) \} dm \nonumber\\
 = & \frac{1}{2} \int_0^t \int_0^{\lambda^{\ast}} \int_0^{\lambda^{\ast}-m} \Theta(m+m^{\ast}) C_n^K(m, m^{\ast}) g_n(m, \tau) g_n(m^{\ast}, \tau) dm^{\ast} dm d\tau \nonumber\\
&- \int_0^t \int_0^{\lambda^{\ast}} \int_0^n \Theta(m) C_n^K(m, m^{\ast}) g_n(m, \tau) g_n(m^{\ast}, \tau) dm^{\ast} dm d\tau \nonumber\\
&+ \int_0^t \int_0^{\lambda^{\ast}} \int_0^{m^{\ast}} \Theta(m) b(m|m^{\ast}) \mathcal{S}^R_n(m^{\ast}) g_n(m^{\ast}, \tau) dm dm^{\ast} d\tau \nonumber\\
&+ \int_0^t \int_{\lambda^{\ast}}^n \int_0^{\lambda^{\ast}} \Theta(m) b(m|m^{\ast}) \mathcal{S}^R_n(m^{\ast}) g_n(m^{\ast}, s) dm dm^{\ast} d\tau \nonumber\\
&- \int_0^t \int_0^{\lambda^{\ast}} \Theta(m) \mathcal{S}_n^R(m) g_n(m, \tau) dm d\tau \nonumber\\
= & \frac{1}{2} \int_0^t \int_0^{\lambda^{\ast}} \int_0^{\lambda^{\ast}} \tilde{\Theta}_{\lambda^{\ast}}(m, m^{\ast}) C_n^K(m, m^{\ast}) g_n(m, \tau) g_n(m^{\ast}, \tau) dm^{\ast} dm d\tau \nonumber\\
&- \int_0^t \int_0^{\lambda^{\ast}} \int_{\lambda^{\ast}-m}^n \Theta(m) C_n^K(m, m^{\ast}) g_n(m, \tau) g_n(m^{\ast}, \tau) dm^{\ast} dm d\tau \nonumber\\
&+ \int_0^t \int_0^{\lambda^{\ast}}  \Pi_{\Theta}(m^{\ast}) \mathcal{S}^R_n(m^{\ast}) g_n(m^{\ast}, \tau) dm^{\ast} d\tau \nonumber\\
&+ \int_0^t \int_{\lambda^{\ast}}^n \int_0^{\lambda^{\ast}} \Theta(m) b(m|m^{\ast}) \mathcal{S}^R_n(m^{\ast}) g_n(m^{\ast}, \tau) dm dm^{\ast} d\tau.
\end{align*}
This finishes the proof of \eqref{1truncatedidentity}. Next set ${\lambda^{\ast}} = n$ into \eqref{1truncatedidentity}, we have
\begin{align}\label{IdentityE2}
\int_0^n \{ g_n(m, t) & - g_n^{in}(m) \} \Theta(m)dm \nonumber\\
=&\frac{1}{2}\int_0^t \int_0^n \int_{0}^n \tilde{\Theta}_{n}(m, m^{\ast}) C^K_n(m, m^{\ast})g_n(m, \tau) g_n(m^{\ast}, \tau)dm^{\ast} dm d\tau \nonumber\\
&- \int_0^t \int_0^n \int_{ n }^{n} \Theta(m)  C^K_n(m, m^{\ast})g_n(m, \tau) g_n(m^{\ast}, \tau)dm^{\ast} dm d\tau \nonumber\\
&+\int_0^t \int_0^{n} \Pi_{\Theta}(m)\mathcal{S}_n^R(m) g_n(m, \tau)dm d\tau \nonumber\\
&+\int_0^t \int_{ n }^{n}  \int_0^{n} \Theta(m) b(m|m^{\ast}) \mathcal{S}_n^R(m^{\ast}) g_n(m^{\ast}, \tau)dm dm^{\ast} d\tau,
\end{align}
From the Lebesgue's dominated convergence theorem, $g_n \in L^1(0, n)$ and the boundedness of $C^K_n$, $\mathcal{S}^R_n$, we deduce that the second and fourth integrals in the right-hand side to \eqref{IdentityE2} are zero. Thus, we get
\begin{align*}
\int_0^{n} \{ g_n(m, t) & - g_n^{in}(m) \} \Theta(m)dm \nonumber\\
=&\frac{1}{2}\int_0^t \int_0^{n} \int_{0}^{n}\tilde{\Theta}_n(m, m^{\ast}) C^K_n(m, m^{\ast})g_n(m, s) g_n(m^{\ast}, s)dm^{\ast} dm ds\nonumber\\
&+\int_0^t \int_0^{n} \Pi_{\Theta}(m) \mathcal{S}_n^R(m) g_n(m, s)dm ds.
\end{align*}
This proves \eqref{truncatedidentity}. Finally, one can be easily proved \eqref{truncated mass} by setting $ \Theta(m) \equiv m \chi_{(0, n)}(m)$ into \eqref{truncatedidentity}. This completes the proof of Lemma \ref{Lemma1}.
\end{proof}
%%%%%%%%%%%%%%%%%%%%%%%%%%%%%%%%%%%%%%%%%%%
%%%%%%%%%%%%%%%%%%%%%%%%%%%%%%%%%%%%%%%%%%%

To prove Theorem \ref{Theorem1} our aim is to apply a weak $L^1$ compactness technique to the family of truncated solutions $\{g_n\}_{n >1}$ which has been adopted from the technique used in the classical work of Stewart \cite{Stewart:1989} and Lauren\c{c}ot \cite{Laurencot:2000On} to demonstrate the existence of weak solutions to continuous CFE and the existence of gelling solutions to continuous coagulation and multiple fragmentation equation, respectively. The first step in this direction is to check the uniformly bound of the family of solutions $\{g_n\}_{n >1}$.
%%%%%%%%%%%%%%%%%%%%%%%%%%%%%%%%%%%%%%%%%%%%
%%%%%%%%%%%%%%%%%%%%%%%%%%%%%%%%%%%%%%%%%%%%
\subsection{Uniform Bound}
\begin{lem}\label{LemmaUniformbound}
Fix $T>0$. Assume \eqref{coagulation kernel}--\eqref{Initialdatacondition} hold. Then for every $n \ge 1$ and $t \in [0, T]$, there exists a positive constant $\mathcal{A}(T)$ (depending on $T$) such that
\begin{align*}
 \int_0^{n} (m^{-2\sigma }+m) g_n(m, t) dm \le \mathcal{A}(T).
\end{align*}
\end{lem}
\begin{proof}
We first take $\Theta(m)=(m+\tau)^{-2\sigma} \chi_{(0, 1)}(m)$, for $\tau \in (0, 1)$, into \eqref{truncatedidentity}. Then, with this suitable $\Theta$,
\begin{align*}
\tilde{\Theta}_n(m, m^{\ast}) = (m+m^{\ast}+\tau)^{-2\sigma} -(m+\tau)^{-2\sigma}-(m^{\ast}+\tau)^{-2\sigma} \le 0,\ \text{for} \ (m, m^{\ast}) \in (0, 1)^2,
\end{align*}
\begin{align*}
\Pi_{\Theta}(m):=\int_0^1 b(m^{\ast}|m)(m^{\ast}+\tau)^{-2\sigma} dm^{\ast}- (m+\tau)^{-2\sigma}  \le \frac{(\gamma+2)}{(1+\gamma-2\sigma)} m^{-(1+\gamma)}.
\end{align*}
Next, using \eqref{Selection Rate}, \eqref{truncated mass}, \eqref{Initialmass} and above estimated values of $\tilde{\Theta}_n$ and $\Pi_{\Theta}$ into
\eqref{truncatedidentity}, we obtain
\begin{align}\label{Uniformbd1}
\int_0^{1} (m+\tau)^{-2\sigma}& \{ g_n(m, t)  - g_n^{in}(m) \}  dm  \le  \frac{(\gamma+2)}{(1+\gamma-2\sigma)} k_3 \int_0^t \int_0^{n}  \varphi(m) g_n(m, s)dm ds\nonumber\\
\le & \frac{(\gamma+2)}{(1+\gamma-2\sigma)} k_3 \varphi(0)  \int_0^t \bigg\{ (1+\tau)^{2\sigma} \int_0^{1} (m+\tau)^{-2\sigma}  g_n(m, s)dm\nonumber\\
 &+ \int_1^{n} m g_n(m, s)dm  \bigg\} ds \nonumber\\
\le & \frac{(\gamma+2)}{(1+\gamma-2\sigma)} k_3 \varphi(0)  \int_0^t \bigg\{ (1+\tau)^{2\sigma} \int_0^{1} (m+\tau)^{-2\sigma}  g_n(m, s)dm + \mathcal{N}_1^{in}  \bigg\} ds.
\end{align}
By applying \eqref{Initialdatafinite} into \eqref{Uniformbd1}, we get
\begin{align*}
\int_0^{1}  ( m+\tau)^{-2\sigma}  g_n(m, t) & dm  \le  \mathcal{Q} + \frac{(\gamma+2)}{(1+\gamma-2\sigma)}  k_3 \varphi(0) \nonumber\\
 & \times \bigg\{ (1+\tau)^{2\sigma}  \int_0^t  \int_0^{1} (m+\tau)^{-2\sigma}  g_n(m, s)dm ds+ \mathcal{N}_1^{in}t \bigg\},
\end{align*}
and the Gronwall's lemma yields that
\begin{align}\label{Uniformbd2}
\int_0^{1}   (m+\tau)^{-2\sigma} g_n(m, t)dm  \le \bigg( \mathcal{Q}+\frac{(\gamma+2)}{(1+\gamma-2\sigma)} k_3 \varphi(0) \mathcal{N}_1^{in}T \bigg) e^{\frac{(\gamma+2)}{(1+\gamma-2\sigma)} k_3 \varphi(0) (1+\tau)^{2\sigma}T}.
\end{align}
As $\tau \to 0$ and then applying Fatou's lemma to \eqref{Uniformbd2}, we obtain
\begin{align}\label{Uniformbd3}
\int_0^{1}   m^{-2\sigma} g_n(m, t)dm  \le \mathcal{A}_1(T),
\end{align}
where
\begin{align*}
\mathcal{A}_1(T) := \bigg( \mathcal{Q}+\frac{(\gamma+2)}{(1+\gamma-2\sigma)} k_3 \varphi(0) \mathcal{N}_1^{in}T \bigg) e^{\frac{(\gamma+2)}{(1+\gamma-2\sigma)} k_3 \varphi(0)T}.
\end{align*}
It follows from \eqref{Uniformbd3}, \eqref{truncated mass} and \eqref{Initialmass} that
\begin{align*}
\int_0^{n} (m+m^{-2\sigma} ) g_n(m, t)dm = & \int_0^{1} m^{-2\sigma}  g_n(m, t)dm + \int_1^{n} m^{-2\sigma}  g_n(m, t)dm\nonumber\\
 &+ \int_0^{n} m  g_n(m, t)dm
\le  \mathcal{A}(T):=\mathcal{A}_1(T) + 2 \mathcal{N}_{1}^{in},
\end{align*}
which finishes the proof of the Lemma \ref{LemmaUniformbound}.
\end{proof}
%%%%%%%%%%%%%%%%%%%%%%%%%%%%%%%%%%%%%%%%
%%%%%%%%%%%%%%%%%%%%%%%%%%%%%%%%%%%%%%%%

We next derive some estimates uniformly with respect to $n \in \mathds{N}$ by using the coagulation kernel \eqref{truncatedcoaker} which are helpful for the next subsections.
%%%%%%%%%%%%%%%%%%%%%%%%%%%%%%%%%%%%%%%
%%%%%%%%%%%%%%%%%%%%%%%%%%%%%%%%%%%%%%%
\begin{lem}\label{Lemma3}
Suppose \eqref{coagulation kernel}--\eqref{Initialdatacondition} hold. Let $T>0$ with $0 < t \le T$. For $\lambda \in ( 1, n) $, then the following estimates are true
\begin{align*}
&(i)\ \  \int_0^t \int_{\lambda}^{n} \int_{\lambda}^{n}   C^K_n(m, m^{\ast})g_n(m, s) g_n(m^{\ast}, s)dm^{\ast} dm ds \le \mathcal{A}^{\dag}(T),\\
%\le  \frac{2\mathcal{N}_1^{in}}{L} \{  2+ N \omega(L) k_4 t\mathcal{A}(T)  \}.
&(ii)\ \  \int_0^t \bigg(  \int_{\lambda}^n \Gamma_n(m) g_n(m, s) dm \bigg)^2 ds \le k_1^{-1} \mathcal{A}^{\dag}(T),
\end{align*}
where \begin{align*}
\mathcal{A}^{\dag}(T) := 2\mathcal{N}_1^{in} \{2/{\lambda}   + k_3 \eta \omega(\lambda)  T \}.
\end{align*}
\end{lem}
\begin{proof}
For $\lambda > 1$, we define $\Theta(m)=m \wedge \lambda$ for $m \in \mathds{R}_{>0}$. The corresponding function $\tilde{ \Theta}_n$ in \eqref{Identity1} satisfies
\[
\tilde{ \Theta}_n(m, m^{\ast}):=\begin{cases}
0,\ & \text{if}\ m+m^{\ast} < \lambda, \ m < \lambda,\ m^{\ast} < \lambda, \\
\lambda-(m+m^{\ast}),\ &  \text{if}\ m+m^{\ast} \ge \lambda,\ m < \lambda,\ m^{\ast} < \lambda,\\
-m, \ &  \text{if}\ m+m^{\ast} \ge \lambda,\ m < \lambda,\ m^{\ast} \ge \lambda,\\
- m^{\ast}, \  &  \text{if}\ m+m^{\ast} \ge \lambda,\ m \ge \lambda,\ m^{\ast} < \lambda,\\
-\lambda,\ &  \text{if}\ m+m^{\ast} \ge \lambda,\ m > \lambda,\ m^{\ast} > \lambda.
\end{cases}
\]
Substituting above values of $\tilde{ \Theta}_n$ and the choice of $\Theta$ into \eqref{truncatedidentity}, we have
\begin{align}\label{EqLemma31}
\int_0^{\lambda}  m \{ g_n(m, t)  - g_n^{in}(m) \}  dm & +\int_{\lambda}^n \lambda \{ g_n(m, t) - g_n^{in}(m) \}  dm \nonumber\\
&+\frac{\lambda}{2}\int_0^t \int_{\lambda}^{n} \int_{\lambda}^{n}  C^K_n(m, m^{\ast})g_n(m, s) g_n(m^{\ast}, s)dm^{\ast} dm ds\nonumber\\
\le &\int_0^t \int_0^{\lambda} \Pi_{\Theta}(m)\mathcal{S}_n^R(m) g_n(m, s)dm ds\nonumber\\
&+\int_0^t \int_{\lambda}^{n} \Pi_{\Theta}(m)\mathcal{S}_n^R(m) g_n(m, s)dm ds.
\end{align}
We infer from \eqref{Selection Rate}, \eqref{TNP}, \eqref{1Identity2}, \eqref{truncated mass}, \eqref{Initialmass} and \eqref{EqLemma31} that
\begin{align*}
\frac{\lambda}{2}\int_0^t \int_{\lambda}^{n} \int_{\lambda}^{n}   & C^K_n(m, m^{\ast})g_n(m, s) g_n(m^{\ast}, s)dm^{\ast} dm ds\nonumber\\
\le & 2 \mathcal{N}_1^{in} + k_3 \eta \varphi(\lambda)  \lambda \int_0^t \int_{\lambda}^{n}  m^{1+\nu} g_n(m, s)dm ds\nonumber\\
\le &  \mathcal{N}_1^{in} \{ 2/\lambda +  k_3 \eta  \varphi(\lambda) t \},
\end{align*}
which completes the proof of the Lemma \ref{Lemma3}~$(i)$. In order to prove the second part of this lemma, we use $C_n^k(m, m^{\ast}) =k_1 \Gamma_n(m) \Gamma_n(m^{\ast})$ from \eqref{coagulation kernel} and \eqref{truncatedcoaker} and inserting it into the first part of Lemma \ref{Lemma3} as
\begin{align*}
\int_0^t \bigg( \int_{\lambda}^{n}  \Gamma_n(m)  g_n(m, s) dm \bigg)^2 ds \le k_1^{-1}\mathcal{A}^{\dag}(T).
\end{align*}
This finishes the proof of the Lemma \ref{Lemma3}.
\end{proof}
%%%%%%%%%%%%%%%%%%%%%%%%%%%%%%%%%%%%%%%
%%%%%%%%%%%%%%%%%%%%%%%%%%%%%%%%%%%%%%%

%%%%%%%%%%%%%%%%%%%%%%%%%%%%%%%%%%%%%%%%%%%%
%%%%%%%%%%%%%%%%%%%%%%%%%%%%%%%%%%%%%%%%%%%%
\begin{lem}\label{Lemma4}
Suppose \eqref{coagulation kernel}--\eqref{Initialdatacondition} hold. Let $T>0$ with $0 < t \le T$. Then followings are true.
\begin{align*}
&(i)\ \  \int_0^t \int_{0}^{n} \int_{0}^{n}   C^K_n(m, m^{\ast})g_n(m, s) g_n(m^{\ast}, s)dm^{\ast} dm ds \le \mathcal{A}_{\dag}(T),\\
%\le  \frac{2\mathcal{N}_1^{in}}{L} \{  2+ N \omega(L) k_4 t\mathcal{A}(T)  \}.
&(ii)\ \  \int_0^t \bigg(  \int_{0}^1 m^{-\sigma}  g_n(m, s) dm \bigg)^2 ds \le k_1^{-1}\mathcal{A}_{\dag}(T),\\
&(iii)\ \  \int_0^t \bigg(  \int_{1}^n  \Gamma_n(m) g_n(m, s) dm \bigg)^2 ds \le k_1^{-1}\mathcal{A}_{\dag}(T).
\end{align*}
\end{lem}
\begin{proof} Set $\Theta \equiv 1$ into \eqref{truncatedidentity}, and the corresponding $\tilde{\Theta}_n$ and $\Pi_{\Theta}$ are
\begin{align}\label{Identity41}
\tilde{\Theta}_n (m, m^{\ast})=-1,
\end{align}
and
\begin{align}\label{Identity42}
\Pi_{\Theta}(m)=\eta-1.
\end{align}
Then, by using \eqref{Selection Rate}, \eqref{Identity41}, \eqref{Identity42}, \eqref{Initialdatafinite}, Lemma \ref{LemmaUniformbound} and \eqref{TNP}, we obtain
\begin{align*}%\label{truncatedidentity11}
\frac{1}{2}\int_0^t \int_0^{n} \int_{0}^{n} & C^K_n(m, m^{\ast})g_n(m, s) g_n(m^{\ast}, s)dm^{\ast} dm ds\nonumber\\
=&\int_0^{n}  g_n(m, t) dm
+\int_0^{n}  g_n^{in}(m) dm+  (\eta-1) k_3 \varphi(0) \int_0^t \int_0^{n} m^{1+\nu} g_n(m, s)dm ds\nonumber\\
\le & \frac{1}{2} \mathcal{A}_{\dag}(T) := \mathcal{A}(T) + \mathcal{Q}+  (\eta-1) k_3 \varphi(0) \mathcal{A}(T)  T.
\end{align*}
Finally, Lemma \ref{Lemma4} $(ii)$ and $(iii)$ can easily be obtained by using \eqref{coagulation kernel} into the first Lemma \ref{Lemma4}.
\end{proof}

%%%%%%%%%%%%%%%%%%%%%%%%%%%%%%%%%%%%%%%%%%%
%%%%%%%%%%%%%%%%%%%%%%%%%%%%%%%%%%%%%%%%%%%
In the next subsection, we discuss the uniform integrability of the family of solutions $\{ g_n \}_{n \in \mathds{N}}$ to apply Dunford Pettis theorem.
 For $n > 1$, $T>0$, $\lambda \in (1, n)$, $\delta \in (0, 1)$, and $t \in [0, T]$, we introduce the following notation:
\[
\Xi_{\lambda, \delta}^n = \sup \begin{cases}
\int_0^{\lambda} \chi_B(m) (1+m^{-\sigma}) g_n(m, t)dm, \\
B\ \text{is any measurable subset of}\ \mathds{R}_{>0}\  \text{with} \ |B| \le \delta.
\end{cases}
\]

%%%%%%%%%%%%%%%%%%%%%%%%%%%%%%%%%%%%%%
%%%%%%%%%%%%%%%%%%%%%%%%%%%%%%%%%%%%%%
\subsection{Uniform Integrability}
%%%%%%%%%%%%%%%%%%%%%%%%%%%%%%%%%%%%%%
%%%%%%%%%%%%%%%%%%%%%%%%%%%%%%%%%%%%%%
\begin{lem}\label{LemmaUniform Integrability}
Let $T \in (0, \infty)$ and $\lambda \in (1, n)$. Then, for every $n > 1$, $t \in [0, T]$ and $\delta \in (0, 1)$, we have
\begin{align}\label{LemmaUniform Integrability1}
  \Xi_{\lambda, \delta}^n (m)  \le & \bigg\{ \Xi_{\lambda, \delta}^n(0)+  k_3 (\gamma+2)  \varphi(0)
   \bigg[ \frac{|\delta|^{\frac{p-1}{p}}}{(p\gamma+1)^{1/p}}  +  \frac{|\delta|^{\frac{p_1-1}{p_1}}}{(p_1(\gamma-\sigma)+1)^{1/p_1} }
   \bigg] \mathcal{A}(T) T  \bigg\} \nonumber\\
   &~~~~~~~~~~~~~~~\times e^{ \frac{1}{2} k_1 \Gamma^2(\lambda) \lambda^{2\sigma} \mathcal{A}(T) T}.
\end{align}
In addition, for every $\epsilon > 0$, there exists a $\lambda_{\epsilon} >1$ (depending on $\epsilon$) such that
\begin{align}\label{LemmaUniform Integrability2}
\int_{\lambda_{\epsilon}}^{\infty} (1+m^{-\sigma}) g_n(m, t) dm < \epsilon.
\end{align}
Furthermore, for every $\delta >0$, there exists a $\epsilon>0$  (depending on $\delta $) for any measurable subset  $B$ of $\mathds{R}_{>0}$ such that $|B| < {\delta}$, then
\begin{align}\label{LemmaUniform Integrability3}
\sup_{n \ge 1 }\int_{B} (1+m^{-\sigma}) g_n(m, t) dm < \epsilon.
\end{align}
\end{lem}
\begin{proof}
Let $\lambda \in (1, n)$. Let $B$ be a measurable subset of $ (0, n)$ such that $|B| \le \delta $. It follows from the non-negativity of $C^K_n$,
 $\mathcal{S}_n^R$, $g_n$, and \eqref{truncatedcoaker}, \eqref{truncatedselrate} and \eqref{Cmfetruncated} that
 \begin{align}\label{Uni Integrability}
\int_0^{\lambda} & \chi_B(m)  (1+m^{-\sigma}) g_n(m, t) dm \le  \int_0^{\lambda} \chi_B(m) (1+m^{-\sigma}) g_n^{in}(m) dm\nonumber\\
 &+ \frac{1}{2} \int_0^t \int_0^{\lambda} \int_0^{m} \chi_B(m) (1+m^{-\sigma}) C_k^n(m-m^{\ast}, m^{\ast}) g_n(m-m^{\ast}, s) g_n(m^{\ast}, s) dm^{\ast} dm ds\nonumber\\
 &+ \int_0^t \int_0^{n} \int_m^n \chi_B(m) (1+m^{-\sigma}) b(m|m^{\ast}) \mathcal{S}_n^R(m^{\ast}) g_n(m^{\ast}, s) dm^{\ast} dm ds.
\end{align}
 Applying Fubini's theorem to the second and last integrals on the right-hand to \eqref{Uni Integrability} and using the transformation $m-m^{\ast}=m'$ and $m^{\ast}= {m^{\ast}}'$, we obtain
\begin{align}\label{Uni Integrability1}
\int_0^{\lambda} & \chi_B(m)  (1+m^{-\sigma}) g_n(m, t) dm \le  \int_0^{\lambda} \chi_B(m) (1+m^{-\sigma}) g_n^{in}(m) dm\nonumber\\
 &+ \frac{1}{2} \int_0^t \int_0^{\lambda} \int_0^{\lambda-m^{\ast}} \chi_B(m+m^{\ast}) (1+(m+m^{\ast})^{-\sigma}) C_k^n(m, m^{\ast}) g_n(m, s) g_n(m^{\ast}, s) dm dm^{\ast} ds\nonumber\\
 &+ \int_0^t \int_0^n \int_0^{m^{\ast}} \chi_B(m) (1+m^{-\sigma}) b(m|m^{\ast}) \mathcal{S}_n^R(m^{\ast}) g_n(m^{\ast}, s) dm dm^{\ast}  ds\nonumber\\
\le &  \Xi_{\lambda, \delta}^n(0) + \frac{1}{2} \int_0^t \int_0^{\lambda} \int_0^{\lambda} \chi_{-m^{\ast}+B}(m) (1+(m+m^{\ast})^{-\sigma}) C_k^n(m, m^{\ast}) g_n(m, s) g_n(m^{\ast}, s) dm^{\ast} dm ds \nonumber\\
 &+ \int_0^t \int_0^n \int_0^{m^{\ast}} \chi_B(m) (1+m^{-\sigma}) b(m|m^{\ast}) \mathcal{S}_n^R(m^{\ast}) g_n(m^{\ast}, s) dm dm^{\ast} ds.
\end{align}
 By using H\"{o}lder's inequality for $p\in (1, 2)$, we estimate the following term as
\begin{align}\label{Uni Integrability2}
 \int_0^{m^{\ast}} \chi_B(m)  (1+m^{-\sigma}) & b(m|m^{\ast}) dm\nonumber\\
  = & \frac{(\gamma+2)}{m^{\ast}{^{\gamma+1}}} \bigg[ \int_0^{m^{\ast}} \chi_B(m) m^{\gamma} dm + \int_0^{m^{\ast}}
   \chi_B(m) m^{\gamma-\sigma} dm \bigg] \nonumber\\
\le & \frac{(\gamma+2)}{m^{\ast}{^{\gamma+1}}} \bigg[ |B|^{\frac{p-1}{p}}  \bigg[ \int_0^{m^{\ast}}  m^{p\gamma} dm
 \bigg]^{\frac{1}{p}}+ |B|^{\frac{p-1}{p}} \bigg[\int_0^{m^{\ast}}  m^{p(\gamma-\sigma)} dm \bigg]^{\frac{1}{p}} \bigg] \nonumber\\
\le & \frac{(\gamma+2)}{{m^{\ast}}^{\gamma+1}} \bigg[ |B|^{\frac{p-1}{p}}  \bigg[ \frac{ m^{p\gamma+1}}{p\gamma+1}
\bigg|_0^{m^{\ast}} \bigg]^{\frac{1}{p}}+ |B|^{\frac{p-1}{p}} \bigg[  \frac{m^{p(\gamma-\sigma)+1}}{p(\gamma-\sigma)+1}\bigg|_0^{m^{\ast}}
 \bigg]^{\frac{1}{p}} \bigg] \nonumber\\
\le & \frac{(\gamma+2)}{{m^{\ast}}^{\gamma+1}} \bigg[ |\delta|^{\frac{p-1}{p}}
 \bigg[ \frac{ {m^{\ast}}^{p\gamma+1}}{p\gamma+1} \bigg]^{\frac{1}{p}}+ |\delta|^{\frac{p-1}{p}} \bigg[  \frac{{m^{\ast}}^{p(\gamma-\sigma)+1}}
 {p(\gamma-\sigma)+1} \bigg]^{\frac{1}{p}} \bigg]\nonumber\\
\le & \frac{(\gamma+2)}{{m^{\ast}}^{\gamma+1}}  \bigg[ \frac{|\delta|^{\frac{p-1}{p}}}{(p\gamma+1)^{1/p}}  {m^{\ast}}^{((1/p)+\gamma)} +  \frac{|\delta|^{\frac{p-1}{p}}}
{(p(\gamma-\sigma)+1)^{1/p} } {m^{\ast}}^{ \frac{1}{p}+\gamma-\sigma}  \bigg] \nonumber\\
\le &  \bigg[ \frac{|\delta|^{\frac{p-1}{p}}}{(p\gamma+1)^{1/p}}   {m^{\ast}}^{1/p-1} +  \frac{|\delta|^{\frac{p-1}{p}}}
{(p(\gamma-\sigma)+1)^{1/p} } {m^{\ast}}^{\frac{1}{p}-1-\sigma}  \bigg].
\end{align}
It follows from \eqref{coagulation kernel} that
\begin{align}\label{Inequalitykernel1}
C^K(m, m^{\ast}) \le k_1 \Gamma^2(\lambda) \lambda^{2\sigma} (m m^{\ast})^{-\sigma},\ \ \ \text{for}\ \ (m, m^{\ast}) =(0, \lambda)^2,
\end{align}
and
\begin{align}\label{Inequalitykernel2}
C^K(m, m^{\ast}) \le k_1 \Gamma(\lambda) \lambda^{\sigma} m^{-\sigma} \Gamma(m^{\ast}),\ \  \text{for}\ \ (m, m^{\ast}) \in (0, \lambda) \times [\lambda, n).
\end{align}
Substituting \eqref{Uni Integrability2}, \eqref{Selection Rate}, \eqref{Inequalitykernel1} and \eqref{Inequalitykernel2} into \eqref{Uni Integrability1}, we obtain
\begin{align}\label{Uni Integrability3}
\int_0^{\lambda} \chi_B(m) & (1+m^{-\sigma}) g_n(m, t) dm  \nonumber\\
 \le & \Xi_{\lambda, \delta}^n(0) + \frac{1}{2} k_1 \Gamma^2(\lambda) \lambda^{2\sigma} \int_0^t \int_0^{\lambda} \int_0^{\lambda} \chi_{-m^{\ast}+B}(m) (1+{m^{\ast}}^{-\sigma})
 (m m^{\ast})^{-\sigma}\nonumber\\
& ~~~~~~~~~\times   g_n(m, s) g_n(m^{\ast}, s) dm^{\ast} dm ds\nonumber\\
 &+  k_3 (\gamma+2)  \varphi(0) \int_0^t \int_0^n   \bigg[ \frac{|\delta|^{\frac{p-1}{p}}}{(p\gamma+1)^{1/p}}   {m^{\ast}}^{1/p+ \gamma} +  \frac{|\delta|^{\frac{p-1}{p}}}{(p(\gamma-\sigma)+1)^{1/p} } {m^{\ast}}^{\frac{1}{p}+\gamma-\sigma }  \bigg] \nonumber\\
 &~~~~~~~~~~~~~~~~~~~~~~ \times  g_n(m^{\ast}, s)  dm^{\ast} ds.
 \end{align}
 Applying Lemma \ref{LemmaUniformbound} into \eqref{Uni Integrability3}, we get
 \begin{align}\label{Uni Integrability31}
 \int_0^{\lambda} \chi_B(m) & (1+m^{-\sigma}) g_n(m, t) dm  \nonumber\\
 \le & \Xi_{\lambda, \delta}^n(0) + \frac{1}{2} k_1 \Gamma^2(\lambda) \lambda^{2\sigma} \int_0^t \int_0^{\lambda} \int_0^{\lambda}
  \chi_{-m^{\ast}+B}(m) (1+{m^{\ast}}^{-\sigma})\nonumber\\
 &~~~~~~~~~~~~~~~~~ \times (1+m^{-\sigma}) {m^{\ast}}^{-\sigma} g_n(m, s) g_n(m^{\ast}, s) dm^{\ast} dm ds\nonumber\\
 &+ (\gamma+2)  k_3  \varphi(0)  \bigg[ \frac{|\delta|^{\frac{p-1}{p}}}{(p\gamma+1)^{1/p}}   +  \frac{|\delta|^{\frac{p-1}{p}}}{(p(\gamma-\sigma)+1)^{1/p} }   \bigg]
 \mathcal{A}(T) t\nonumber\\
 \le & \Xi_{\lambda, \delta}^n(0) + \frac{1}{2} k_1 \Gamma^2(\lambda) \lambda^{2\sigma} \mathcal{A}(T) \int_0^t \int_0^{\lambda} \chi_{B}(m)
  (1+m^{-\sigma})  g_n(m, s) dm ds\nonumber\\
 &+ (\gamma+2)  k_3  \varphi(0) \bigg[ \frac{|\delta|^{\frac{p-1}{p}}}{(p\gamma+1)^{1/p}}   +  \frac{|\delta|^{\frac{p-1}{p}}}{(p(\gamma-\sigma)+1)^{1/p} }   \bigg]
  \mathcal{A}(T) T.
\end{align}
Finally, applying Gronwall's inequality, we find
\begin{align}\label{Uni Integrability4}
\int_0^{\lambda} \chi_B(m) & (1+m^{-\sigma}) g_n(m, t) dm  \le \bigg\{ \Xi_{\lambda, \delta}^n(0)+  (\gamma+2)  k_3 \varphi(0)
   \bigg[ \frac{|\delta|^{\frac{p-1}{p}}}{(p\gamma+1)^{1/p}} \nonumber\\  & +  \frac{|\delta|^{\frac{p-1}{p}}}{(p(\gamma-\sigma)+1)^{1/p} }
   \bigg] \mathcal{A}(T) T  \bigg\}  e^{ \frac{1}{2} k_1 \Gamma^2(\lambda) \lambda^{2\sigma} \mathcal{A}(T) T}.
\end{align}
This finishes the proof of \eqref{LemmaUniform Integrability}. The proof of the second part of Lemma \ref{LemmaUniform Integrability} follows from \eqref{truncated mass} and Lemma \ref{LemmaUniformbound}. In order to prove \eqref{LemmaUniform Integrability3}, the following term is estimated, by using \eqref{LemmaUniform Integrability1} and \eqref{LemmaUniform Integrability1}, as
\begin{align}\label{Uni Integrability5}
\int_{B}  (1+m^{-\sigma}) g_n(m, t) dm \le & \int_{0}^{\lambda_{\epsilon}} \chi_B(m) (1+m^{-\sigma}) g_n(m, t) dm +\int_{\lambda_{\epsilon}}^n (1+m^{-\sigma}) g_n(m, t) dm \nonumber\\
\le &  \bigg\{ \Xi_{\lambda_{\epsilon}, \delta}^n(0)+  (\gamma+2)  k_3 \varphi(0)
   \bigg[ \frac{|\delta|^{\frac{p-1}{p}}}{(p\gamma+1)^{1/p}} \nonumber\\  & +  \frac{|\delta|^{\frac{p-1}{p}}}{(p(\gamma-\sigma)+1)^{1/p} }
   \bigg] \mathcal{A}(T) T  \bigg\}  e^{ \frac{1}{2} k_1 \Gamma^2(\lambda) \lambda^{2\sigma} \mathcal{A}(T) T} +\epsilon.
\end{align}
Next, by applying Cauchy Schwarz inequality and \eqref{Initialdatacondition}, we estimate the following term as
\begin{align*}%\label{Uni Integrability51}
\Xi_{\lambda_{\epsilon}, \delta}^n(0) \le & \bigg(\int_{0}^{\lambda_{\epsilon}} |\chi_B(m) {g^{in}_n(m)}^{\frac{1}{2}} |^2 dm \bigg)^{\frac{1}{2}}   \bigg(\int_{0}^{\lambda_{\epsilon}} | (1+m^{-\sigma}) {g^{in}_n(m)}^{\frac{1}{2}} |^2 dm \bigg)^{\frac{1}{2}} \nonumber\\
\le & |\delta| \bigg(\int_{0}^{\lambda_{\epsilon}} g^{in}_n(m) dm \bigg)^{\frac{1}{2}}   \bigg(\int_{0}^{\lambda_{\epsilon}} (1+m^{-2\sigma}+2m^{-\sigma}) g^{in}_n(m) dm \bigg)^{\frac{1}{2}} \le  2 |\delta| \ \mathcal{Q}.
\end{align*}
As $\delta \to 0$, we have
\begin{align}\label{Uni Integrability6}
 \lim_{\delta \to 0} \sup_{n>1} \Xi_{\lambda, \delta}^n(0) =0.
\end{align}
 Using \eqref{Uni Integrability6} into \eqref{Uni Integrability5} for $\delta \to 0$, we finally obtain
\begin{align*}
\int_{B}  (1+m^{-\sigma}) g_n(m, t) dm < \epsilon.
\end{align*}
This completes the proof.
\end{proof}
%%%%%%%%%%%%%%%%%%%%%%%%%%%%%%%%%%%%%%%%%%%%%%%%%
%%%%%%%%%%%%%%%%%%%%%%%%%%%%%%%%%%%%%%%%%%%%%%%%%

%%%%%%%%%%%%%%%%%%%%%%%%%%%%%%%%%%%%%%%%%%%%%%%
%%%%%%%%%%%%%%%%%%%%%%%%%%%%%%%%%%%%%%%%%%%%%%%
\subsection{Time Equi-continuity}
%%%%%%%%%%%%%%%%%%%%%%%%%%%%%%%%%%%%%%%%%%%%%%%
%%%%%%%%%%%%%%%%%%%%%%%%%%%%%%%%%%%%%%%%%%%%%%%
Next, our aim being to apply a refined version of Arzel\`{a}-Ascoli's theorem, we have to show that there exists a subsequence of the family of $\{g_n\}_{n>1}$ converges weakly in the topology $L^1_{-\sigma, 1}(\mathds{R}_{>0}).$ For that purpose, we need to prove the following equi-continuity result.
\begin{lem}\label{LemmaEquicontinuity}
Assume \eqref{coagulation kernel}--\eqref{Initialdatacondition} hold. Let $T>0$ and $\lambda>1$, then there is a positive constant $\mathcal{A}^{\ast}(\lambda, T)$ depending on $\lambda$ and $T$
 such that
\begin{align*}
\bigg|\int_0^{\infty}  (m^{-\sigma}+1) \{g_n(m, t)-g_n(m, s) \}dm \bigg|\le \mathcal{A}^{\ast}(\lambda, T)(t-s),
\end{align*}
for every $n > 1$, and $0 \le s \le t \le T$.
\end{lem}
\begin{proof}
We set $\omega(m) := ((m+\theta)^{-\sigma}+1) \chi_{(0, \lambda)}(m) \text{sign}( g_n(m, t)-g_n(m, s))$, for $\theta \in (0, 1)$, and $\lambda= \lambda^{\ast}$ into \eqref{1truncatedidentity}. Then, we obtain
 \begin{align}\label{Equicontinuity1}
\int_0^{\lambda}   ((m+\theta)^{-\sigma} & +1)   |g_n(m, t)-g_n(m, s)|dm\nonumber\\
\le &\frac{1}{2}\int_s^t \int_0^{\lambda} \int_{0}^{\lambda} |\tilde{\Theta}_{\lambda}(m, m^{\ast})| C^K_n(m, m^{\ast})g_n(m, \tau) g_n(m^{\ast}, \tau)dm^{\ast} dm d\tau \nonumber\\
&+ \int_s^t \int_0^{\lambda} \int_{ \lambda }^{n}((m+\theta)^{-\sigma}+1)C^K_n(m, m^{\ast})g_n(m, \tau) g_n(m^{\ast}, \tau)dm^{\ast} dm d\tau \nonumber\\
&+\int_s^t \int_0^{\lambda} | \Pi_{\Theta}(m)| \mathcal{S}_n^R(m) g_n(m, \tau)dm d\tau \nonumber\\
&+\int_s^t \int_{ \lambda }^{n}  \int_0^{\lambda} ((m+\theta)^{-\sigma}+1) b(m|m^{\ast}) \mathcal{S}_n^R(m^{\ast}) g_n(m^{\ast}, \tau)dm dm^{\ast} d\tau.
\end{align}
 Next, we estimate $\tilde{\Theta}_{\lambda}$ and $\Pi_{\Theta}$ as
\begin{align*}
| \tilde{\Theta}_{\lambda}(m, m^{\ast}) | \le  & |((m+m^{\ast}+\theta)^{-\sigma}+1)+ ((m+\theta)^{-\sigma}+1)+ ((m^{\ast}+\theta)^{-\sigma}+1)| \nonumber\\
<  & 3 \{ 1+m^{-\sigma}+{m^{\ast}}^{-\sigma}  \}.
\end{align*}
and
\begin{align*}
|\Pi_{\Theta}(m)| \le & \int_0^m b(m^{\ast}|m)({m^{\ast}}^{-\sigma}+1) dm^{\ast}+ (m^{-\sigma}+1) \nonumber\\
= & \frac{(\gamma+2)} {m^{\gamma+1}} \int_0^m {m^{\ast}}^{\gamma} ({m^{\ast}}^{-\sigma}+1) dm^{\ast}+ (m^{-\sigma}+1) \nonumber\\
= & \frac{(\gamma+2)} {m^{\gamma+1}} \bigg( \frac{m^{\gamma-\sigma+1}}{(\gamma-\sigma+1)} + \frac{m^{\gamma+1}}{\gamma+1} \bigg)+ (m^{-\sigma}+1) \nonumber\\
\le & \frac{(\gamma+2)} {(\gamma-\sigma+1)} ( m^{-\sigma} + 1 )+ (m^{-\sigma}+1) = \frac{(3+2\gamma-\sigma)} {(\gamma-\sigma+1)} ( m^{-\sigma} + 1 ).
\end{align*}
Using above values of $| \tilde{\Theta}_{\lambda} |$, $|\Pi_{\Theta}|$, \eqref{coagulation kernel}, \eqref{Selection Rate}, \eqref{Inequalitykernel1} and \eqref{Inequalitykernel2} into \eqref{Equicontinuity1}, we find
  \begin{align}\label{Equicontinuity2}
\int_0^{\lambda}  & ((m+\theta)^{-\sigma}+1)   |g_n(m, t)-g_n(m, s)|dm\nonumber\\
\le &\frac{3}{2} k_1 \Gamma^2(\lambda) \lambda^{2\sigma} \int_s^t \int_0^{\lambda} \int_{0}^{\lambda} \{ 1+m^{-\sigma}+{m^{\ast}}^{-\sigma}  \}
  {(mm^{\ast})}^{-\sigma} g_n(m, \tau) g_n(m^{\ast}, \tau)dm^{\ast} dm d\tau \nonumber\\
&+ k_1 \Gamma(\lambda) \lambda^{\sigma} \int_s^t \int_0^{\lambda} \int_{ \lambda }^{n}(m^{-\sigma}+1) \Gamma(m^{\ast}) m^{-\sigma}  g_n(m, \tau)
 g_n(m^{\ast}, \tau)dm^{\ast} dm d\tau\nonumber\\
&+k_3 \frac{(3+2\gamma-\sigma)} {(\gamma-\sigma+1)} \varphi(0) \int_s^t \int_0^{\lambda} ({m}^{-\sigma}+1) m^{1+\gamma} g_n(m, \tau)dm d\tau \nonumber\\
&+  k_3 \varphi(1)  (\gamma+2)  \int_s^t \int_{ \lambda }^{n} \int_0^{\lambda}  (m^{-\sigma}+1)  m^{\gamma} g_n(m^{\ast}, \tau) dm dm^{\ast} d\tau.
\end{align}
Applying Lemma \ref{LemmaUniformbound} and Lemma \ref{Lemma3} to \eqref{Equicontinuity2}, we obtain
\begin{align}\label{Equicontinuity3}
\int_0^{\lambda}  & ((m+\theta)^{-\sigma}+1)   |g_n(m, t)-g_n(m, s)|dm\nonumber\\
\le &\frac{9}{2}k_1 \Gamma^2(\lambda) \lambda^{2\sigma}  \mathcal{A}^2(T)(t-s) +2 k_1 \Gamma(\lambda) \lambda^{\sigma}  \mathcal{A}(T) \int_s^t  \int_{ \lambda }^{n} \Gamma_n(m^{\ast})
 g_n(m^{\ast}, \tau)dm^{\ast}  d\tau \nonumber\\
&+ 2 k_3 \frac{(3+2\gamma-\sigma)} {(\gamma-\sigma+1)}  \varphi(0)  \mathcal{A}(T)(t-s)  + 2 k_3 \varphi(1) \frac{ (\gamma+2)}{(\gamma-\sigma+1)} \lambda^{1+\gamma} \mathcal{A}(T)  (t-s)\nonumber\\
\le &\frac{9}{2}k_1 \Gamma^2(\lambda) \lambda^{2\sigma}  \mathcal{A}^2(T)(t-s) + 2k_1 \Gamma(\lambda) \lambda^{\sigma}  \mathcal{A}(T) \sqrt{k_1^{-1} \mathcal{A}^{\dag}(T)T}(t-s)\nonumber\\
&+ 2 k_3\mathcal{A}(T) \bigg( \varphi(0) \frac{(3+2\gamma-\sigma)} {(\gamma-\sigma+1)}    +   \varphi(1) \frac{ (\gamma+2)}{(\gamma-\sigma+1)} \lambda^{1+\gamma}  \bigg)(t-s)\nonumber\\
=& \mathcal{A}^{\ast}(\lambda, T) (t-s),
\end{align}
where
\begin{align*}
\mathcal{A}^{\ast}(\lambda, T) := & \frac{9}{2}k_1 \Gamma^2(\lambda) \lambda^{2\sigma}  \mathcal{A}^2(T) + 2k_1 \Gamma(\lambda) \lambda^{\sigma} \mathcal{A}(T) \sqrt{k^{-1} \mathcal{A}^{\dag}(T)T}\nonumber\\
&+ 2 k_3\mathcal{A}(T) \bigg( \varphi(0) \frac{(3+2\gamma-\sigma)} {(\gamma-\sigma+1)}    +   \varphi(1) \frac{ (\gamma+2)}{(\gamma-\sigma+1)} \lambda^{1+\gamma}  \bigg).
 \end{align*}
 As $\theta \in (0, 1)$ is arbitrarily small, then by Fatou's lemma for $\theta \to 0$ and from \eqref{Equicontinuity3}, we obtain
   \begin{align}\label{Equicontinuity4}
\int_0^{\lambda}   (m^{-\sigma}+1)   |g_n(m, t)-g_n(m, s)|dm \le \mathcal{A}^{\ast}(\lambda, T) (t-s).
\end{align}
By using \eqref{Equicontinuity4}, we evaluate the following integral as
\begin{align*}%\label{Equicontinuity5}
\bigg|\int_0^{\infty} & (m^{-\sigma}+1)   [g_n(m, t)-g_n(m, s)]dm \bigg| \nonumber \\
\le & \int_0^{\lambda} (m^{-\sigma}+1)   |g_n(m, t)-g_n(m, s)|dm
+ \int_{\lambda}^{\infty}  (m^{-\sigma}+1)   |g_n(m, t)-g_n(m, s)|dm  \nonumber\\
\le & \mathcal{A}^{\ast}(\lambda, T) (t-s) +2 \epsilon.
\end{align*}
As $\epsilon $ is arbitrary we finally obtain the desire result.
\end{proof}
From Lemma \ref{LemmaEquicontinuity}, the family of truncated solutions $\{ g_n\}_{n > 1}$ is strongly equi-continuous in the topology $L^1_{-\sigma, 1}( \mathds{R}_{>0})$. This implies that the family $\{ g_n\}_{n > 1}$ is also weakly equi-continuous in the same topology $L^1_{-\sigma, 1}( \mathds{R}_{>0})$. Then according to a variant of the Arzel\`{a}-Ascoli theorem \cite{Stewart:1989, Vrabie:1995}, Lemma \ref{LemmaUniform Integrability} and Dunford-Pettis theorem \cite{Stewart:1989} that the family $\{ g_n\}_{n >1}$ is weakly compact in $L^1_{-\sigma, 1}( \mathds{R}_{>0})$ for every $t \in [0, T]$. This implies that there exists a subsequence of $\{ g_n\}_{n > 1}$ (not relabeled) such that
\begin{align}\label{Convergenceweakly1}
g_n \to g \ \ \ \text{in} \ \ \mathcal{C} ( [0, T]; w-L^1_{-\sigma, 1} ( \mathds{R}_{>0}) ).
\end{align}
 As $g(\cdot, t)$ is a weak limit of non-negative functions $\{ g_n\}_{n > 1}$, this implies that $g(\cdot, t) \ge 0$ a.e. in $\mathds{R}_{>0}$ for every $t \in [0, \infty)$. Next our aim is to show that
\begin{align}\label{Convergencestrongly1}
 g  \in   \mathcal{C} ( [0, T]; L^1_{-\sigma, 1} ( \mathds{R}_{>0}) ).
\end{align}
In order to prove \eqref{Convergencestrongly1}, let $s, t \in [0, \infty)$ and $\epsilon >0$. From \eqref{Convergenceweakly1}, as $g_n\rightharpoonup g$  in $L^1_{-\sigma, 1} ( \mathds{R}_{>0}) $, then we have $g_n(t) -g_n(s)$ converges weakly to $g(t) -g(s)$ in $L^1_{-\sigma, 1} ( \mathds{R}_{>0}) $. Next
\begin{align*}
\|g(t) -g(s) \|_{L^1_{-\sigma, 1}( \mathds{R}_{>0})} = \int_0^{\infty} (1-m^{-\sigma}) |g(m, t) -g(m, s) |dm < \epsilon,
\end{align*}
as
\begin{align*}
|t -s|< \delta = \frac{\epsilon} {\mathcal{A}^{\ast}(\lambda, T)},
\end{align*}
which proves \eqref{Convergenceweakly1}. \\

Next, we state some consequences of \eqref{Convergencestrongly1} which are useful to show the convergence of integral operators. These consequences follow from Lemma \ref{Lemma3}, Lemma \ref{Lemma4} and \eqref{Convergencestrongly1} that
\begin{align}\label{Consequence1}
\int_0^t \int_{\lambda}^{\infty} \int_{\lambda}^{\infty}   C^K(m, m^{\ast})g(m, s) g(m^{\ast}, s)dm^{\ast} dm ds \le \mathcal{A}^{\dag}(T),
\end{align}
\begin{align}\label{Consequence2}
 \int_0^t \bigg(  \int_{\lambda}^{\infty}  \Gamma(m) g(m, s) dm \bigg)^2 ds \le k_1^{-1} \mathcal{A}^{\dag}(T),
\end{align}
\begin{align}\label{Consequence3}
 \int_0^t \int_{0}^{\infty} \int_{0}^{\infty}   C^K(m, m^{\ast})g(m, s) g(m^{\ast}, s)dm^{\ast} dm ds \le \mathcal{A}_{\dag}(T),
\end{align}
\begin{align}\label{Consequence4}
\int_0^t \bigg(  \int_{0}^{1}  m^{-\sigma} g(m, s) dm \bigg)^2 ds \le k_1^{-1} \mathcal{A}_{\dag}(T),
\end{align}
and
\begin{align}\label{Consequence5}
\int_0^t \bigg(  \int_{1}^{\infty}  \Gamma(m) g(m, s) dm \bigg)^2 ds \le k_1^{-1} \mathcal{A}_{\dag}(T).
\end{align}
In addition, from \eqref{Initialmass} and \eqref{Convergencestrongly1}, we have
\begin{align*}
\int_0^{\infty} m g(m, t) dm \le \int_0^{\infty} m g^{in}(m) dm.
\end{align*}
This proves \eqref{Initialmass}. Finally, it follows from \eqref{Consequence3}, Fubini's theorem, \eqref{Selection Rate}, \eqref{Breakage function1}, \eqref{Convergencestrongly1} that
 \begin{align*}
 \int_0^{\infty} \int_0^{\infty} C^K(m, m^{\ast}) g(m, t) g(m^{\ast}, t) dm^{\ast} dm  \in L^1(0, t),
\end{align*}
and
\begin{align*}
 \int_0^{\infty} \int_m^{\infty} b(m|m^{\ast}) \mathcal{S}^R(m^{\ast}) g(m^{\ast}, t) dm^{\ast} dm  \in L^1(0, t).
\end{align*}
 We are now in a position to complete the proof of Theorem \ref{Theorem1} in the next subsection.

%%%%%%%%%%%%%%%%%%%%%%%%%%%%%%%%%%%%%%%%%%%%%%%
%%%%%%%%%%%%%%%%%%%%%%%%%%%%%%%%%%%%%%%%%%%%%%%
\subsection{Convergence of integrals}
%%%%%%%%%%%%%%%%%%%%%%%%%%%%%%%%%%%%%%%%%%%%%%
%%%%%%%%%%%%%%%%%%%%%%%%%%%%%%%%%%%%%%%%%%%%%%
In this section, we check that the function $g$ is indeed a solution to \eqref{Cmfe}--\eqref{Initialdata}.
Next consider $\lambda \in (1, \infty)$. For $n > 1$ and $ s \in (0, t)$, we define the following operators:
\begin{align*}
\mathcal{G}_{n, 1} (\lambda, s):= &\int_0^{\lambda} \int_0^{\lambda} \tilde{\Theta}_n(m, m^{\ast}) C^K_n(m, m^{\ast}) g_n(m, s) g_n(m^{\ast}, s) dm^{\ast} dm,\\
\mathcal{G}_{n, 2} (\lambda, s):= & 2\int_0^{\lambda} \int_{\lambda}^n \tilde{\Theta}_n(m, m^{\ast}) C^K_n(m, m^{\ast}) g_n(m, s) g_n(m^{\ast}, s) dm^{\ast} dm,\\
\mathcal{G}_{n, 3} (\lambda, s):= &\int_{\lambda}^n \int_{\lambda}^n \tilde{\Theta}_n(m, m^{\ast}) C^K_n(m, m^{\ast}) g_n(m, s) g_n(m^{\ast}, s) dm^{\ast} dm,\\
\mathcal{G}_{n, 4} (\lambda, s):= &\int_0^{\lambda} \Pi_{\Theta}(m) \mathcal{S}_n^R(m)  g_n(m, s) dm,\\
\mathcal{G}_{n, 5} (\lambda, s):= &\int_{\lambda}^n \int_0^{\lambda} {\Theta}(m) b(m|m^{\ast}) \mathcal{S}_n^R(m^{\ast}) g_n(m^{\ast}, s) dm dm^{\ast},
\end{align*}
and
\begin{align*}
\mathcal{G}_{1} (\lambda, s):= &\int_0^{\lambda} \int_0^{\lambda} \tilde{\Theta}(m, m^{\ast}) C^K(m, m^{\ast}) g(m, s) g(m^{\ast}, s) dm^{\ast} dm,\\
\mathcal{G}_{2} (\lambda, s):= & 2\int_0^{\lambda} \int_{\lambda}^{\infty} \tilde{\Theta}(m, m^{\ast}) C^K(m, m^{\ast}) g(m, s) g(m^{\ast}, s) dm^{\ast} dm,\\
\mathcal{G}_{3} (\lambda, s):= &\int_{\lambda}^{\infty} \int_{\lambda}^{\infty} \tilde{\Theta}(m, m^{\ast}) C^K(m, m^{\ast}) g(m, s) g(m^{\ast}, s) dm^{\ast} dm,\\
\mathcal{G}_{4} (\lambda, s):= &\int_0^{\lambda} \Pi_{\Theta}(m) \mathcal{S}^R(m)  g(m, s) dm,\\
\mathcal{G}_{5} (\lambda, s):= &\int_{\lambda}^{\infty} \int_0^{\lambda} {\Theta}(m) b(m|m^{\ast}) \mathcal{S}^R(m^{\ast}) g(m^{\ast}, s) dm dm^{\ast},
\end{align*}
where
\begin{align}\label{Identity400}
\tilde{\Theta}(m, m^{\ast}):=\Theta(m+m^{\ast})-\Theta(m)-\Theta(m^{\ast}),
\end{align}
and $\Theta \in L^{\infty}(\mathds{R}_{>0})$. For $n > \lambda $, we have $C_n^K = C^K$ in $(0, \lambda]^2$. Then for each $s \in (0, t)$, from \cite[Lemma 2. 9] {Laurencot:2000On} and the weak convergence \eqref{Convergenceweakly1}, we obtain
 \begin{align}\label{convergence11}
 \lim_{n \to \infty} \mathcal{G}_{n, 1} (\lambda, s)  =   \mathcal{G}_{1} (\lambda, s).
 \end{align}
Then by using \eqref{convergence11}, Lemma \ref{LemmaUniformbound} and the Lebesgue's dominated convergence theorem, we have
 \begin{align}\label{convergence12}
 \lim_{n  \to \infty} \int_0^t \mathcal{G}_{n, 1} (\lambda, s) ds =  \int_0^t \mathcal{G}_{1} (\lambda, s) ds.
 \end{align}
Now let us estimate the following term, by using \eqref{coagulation kernel}, as
\begin{align}\label{convergence130}
\int_0^t | \mathcal{G}_{n, 2} & (\lambda, s) + \mathcal{G}_{n, 3} (\lambda, s)  | ds \nonumber\\
 \le & 2 \int_0^t \int_0^{n} \int_{\lambda}^n | \tilde{\Theta}_n(m, m^{\ast}) | C_n^K(m, m^{\ast}) g_n(m, s)
g_n(m^{\ast}, s) dm^{\ast} dm ds \nonumber\\
\le  & 6 \| {\Theta} \|_{L^{\infty}(\mathds{R}_{>0})}  \int_0^t \int_0^{1} \int_{\lambda}^n  C_n^K(m, m^{\ast}) g_n(m, s) g_n(m^{\ast}, s) dm^{\ast} dm ds \nonumber\\
& + 6 \| {\Theta} \|_{L^{\infty}(\mathds{R}_{>0})} \int_0^t \int_1^{n} \int_{\lambda}^n  C_n^K(m, m^{\ast}) g_n(m, s) g_n(m^{\ast}, s) dm^{\ast} dm ds \nonumber\\
\le  & 6 k_1 \| {\Theta} \|_{L^{\infty}(\mathds{R}_{>0})}  \int_0^t \int_0^{1} \int_{\lambda}^n  m^{-\sigma} \Gamma(m^{\ast}) g_n(m, s) g_n(m^{\ast}, s) dm^{\ast} dm ds \nonumber\\
& + 6 k_1 \| {\Theta} \|_{L^{\infty}(\mathds{R}_{>0})} \int_0^t \int_1^{n} \int_{\lambda}^n  \Gamma(m) \Gamma(m^{\ast}) g_n(m, s) g_n(m^{\ast}, s) dm^{\ast} dm ds.
%\le  & 6 k_1 \| {\Theta} \|_{L^{\infty}(\mathds{R}_{>0})} \mathcal{A}(T) \mathcal{A}^{\dag}(T)t
 %+ 6 k_1 \| {\Theta} \|_{L^{\infty}(\mathds{R}_{>0})}  \mathcal{A}(T) \mathcal{A}^{\dag}(T) \mathcal{A}^{\ddag}(T)t.
 \end{align}
 Applying Young's inequality to \eqref{convergence130} and then using Lemma \ref{Lemma3}, Lemma \ref{Lemma4}, and Lemma \ref{LemmaUniformbound}, we estimate \begin{align}\label{convergence13}
\int_0^t | \mathcal{G}_{n, 2} & (\lambda, s) + \mathcal{G}_{n, 3} (\lambda, s)  | ds \nonumber\\
\le  & 6 k_1 \| {\Theta} \|_{L^{\infty}(\mathds{R}_{>0})} \mathcal{A}(T) \int_0^t  \int_{\lambda}^n  \Gamma(m^{\ast}) g_n(m^{\ast}, s) dm^{\ast} ds \nonumber\\
& + 6 k_1 \| {\Theta} \|_{L^{\infty}(\mathds{R}_{>0})} \int_0^t  \bigg( \int_1^{n} \Gamma(m) g_n(m, s) dm \bigg)  \bigg( \int_{\lambda}^n  \Gamma(m^{\ast})  g_n(m^{\ast}, s) dm^{\ast} \bigg) ds \nonumber\\
\le  & 6 k_1 \| {\Theta} \|_{L^{\infty}(\mathds{R}_{>0})} \mathcal{A}(T)  \bigg[\int_0^t \bigg( \int_{\lambda}^n  \Gamma(m^{\ast}) g_n(m^{\ast}, s) dm^{\ast} \bigg)^2 ds \bigg]^{\frac{1}{2}} \bigg[\int_0^t 1^2 ds \bigg]^{\frac{1}{2}} \nonumber\\
& + 6 k_1 \| {\Theta} \|_{L^{\infty}(\mathds{R}_{>0})} \bigg[\int_0^t  \bigg( \int_1^{n} \Gamma(m) g_n(m, s) dm \bigg)^2 ds \bigg]^{\frac{1}{2}} \nonumber\\
&~~~~~~~~\times \bigg[\int_0^t \bigg( \int_{\lambda}^n  \Gamma(m^{\ast})  g_n(m^{\ast}, s) dm^{\ast} \bigg)^2 ds  \bigg]^{\frac{1}{2}} \nonumber\\
\le  & 6 k_1 \| {\Theta} \|_{L^{\infty}(\mathds{R}_{>0})} \mathcal{A}(T) T^{\frac{1}{2}}  \sqrt{k_1^{-1}\mathcal{A}^{\dag}(T)} + 6  \| {\Theta} \|_{L^{\infty}(\mathds{R}_{>0})}  \sqrt{\mathcal{A}^{\dag}(T)~ \mathcal{A}_{\dag}(T)}.
 \end{align}
Similar to \eqref{convergence13}, the following can be estimated as
\begin{align}\label{convergence14}
\int_0^t | \mathcal{G}_{2} (\lambda, s) & + \mathcal{G}_{3} (\lambda, s)  | ds \nonumber\\
\le  & 6  \| {\Theta} \|_{L^{\infty}(\mathds{R}_{>0})} \bigg( k_1\| g\|_{L^1_{-2\sigma, 1}(\mathds{R}_{>0})} T^{\frac{1}{2}}  \sqrt{k_1^{-1}\mathcal{A}^{\dag}(T)} +   \sqrt{\mathcal{A}^{\dag}(T)~ \mathcal{A}_{\dag}(T)} \bigg).
 \end{align}
One can see that $\mathcal{G}_{n, 1} (\lambda, s)+\mathcal{G}_{n, 2} (\lambda, s) + \mathcal{G}_{n, 3} (\lambda, s) $ and
$\mathcal{G}_{1} (\lambda, s)+\mathcal{G}_{2} (\lambda, s) + \mathcal{G}_{3} (\lambda, s) $ do not depend on $\lambda \in (1, \infty)$. The above inequality is valid for every  $\lambda \in (1, \infty)$, we finally obtain
\begin{align}\label{convergence15}
\lim_{n \to \infty}\int_0^t & \int_0^{n} \int_0^{n} \tilde{\Theta}_n(m, m^{\ast}) C^K_n(m, m^{\ast}) g_n(m, s) g_n(m^{\ast}, s) dm^{\ast} dm\nonumber\\
=& \int_0^t  \int_0^{\infty} \int_0^{\infty} \tilde{\Theta}(m, m^{\ast}) C^K(m, m^{\ast}) g(m, s) g(m^{\ast}, s) dm^{\ast} dm.
\end{align}
It remains to pass to the limit in the multiple fragmentation part. Next, for each $n > \lambda $, we have $\mathcal{S}_n^R \le \mathcal{S}^R $ in $(0, \lambda)$. Using \eqref{TNP} and \eqref{Selection RateAdditional}, we obtain
\begin{align*}
| \Pi_{\Theta}(m) \mathcal{S}_n^R(m^{\ast})| = & \bigg| \bigg( \int_0^m b(m^{\ast}| m) \Theta(m^{\ast}) dm^{\ast} -\Theta(m) \bigg) \mathcal{S}_n^R(m^{\ast}) \bigg| \nonumber\\
\le  & k_3 \varphi(0)  \|\Theta \|_{L^{\infty} \mathds{R}_{>0} } (\eta +1 ) \Sigma, \ \ \ a.e.\ \text{in}\ \ (0, \lambda).
\end{align*}
It then follows from the weak convergence $g_n \rightharpoonup g$ that
\begin{align}\label{convergence16}
\lim_{n \to \infty} \int_0^t \int_0^{\lambda} \Pi_{\Theta}(m) \mathcal{S}_n^R(m)  g_n(m, s) dm ds =
\int_0^t \int_0^{\lambda} \Pi_{\Theta}(m) \mathcal{S}^R(m)  g(m, s) dm ds.
\end{align}
Finally, we evaluate the last integral $\mathcal{G}_{n, 5} (\lambda, s)$, by using \eqref{Selection Rate}, \eqref{TNP} and \eqref{Initialmass}, as
\begin{align}\label{convergence17}
\int_0^t \int_{\lambda}^n \int_0^{\lambda} &  {\Theta}(m) b(m|m^{\ast}) \mathcal{S}^R(m^{\ast})  g_n(m^{\ast}, s) dm dm^{\ast} ds \nonumber\\
\le & k_3 \|\Theta \|_{L^{\infty} (\mathds{R}_{>0} ) } (\eta +1 ) \varphi(\lambda)  \int_0^t \int_{\lambda}^n m^{1+\gamma}  g_n(m, s) dm ds\nonumber\\
\le & k_3 \|\Theta \|_{L^{\infty} (\mathds{R}_{>0} ) } (\eta +1 ) \varphi(\lambda)  \int_0^t \int_{\lambda}^n m  g_n(m, s) dm ds\nonumber\\
\le & k_3 \|\Theta \|_{L^{\infty} (\mathds{R}_{>0} )} (\eta +1 ) \varphi(\lambda) \mathcal{N}_1^{in} t.
\end{align}
Similarly, by using \eqref{Selection Rate}, \eqref{TNP} and \eqref{Initialmass}, we estimate
\begin{align}\label{convergence18}
\int_0^t \int_{\lambda}^{\infty} \int_0^{\lambda}   {\Theta}(m) b(m|m^{\ast})  g(m^{\ast}, s) dm dm^{\ast} ds
\le  k_3 \|\Theta \|_{L^{\infty} (\mathds{R}_{>0} ) } (\eta +1 ) \varphi(\lambda) \mathcal{N}_1^{in} t.
\end{align}
Thanks to $\varphi$, and as $\lambda \to \infty$, one can see that \eqref{convergence17} and \eqref{convergence18} go to $0$. Thus, it follows from \eqref{convergence16}, \eqref{convergence17} and \eqref{convergence18} that
\begin{align}\label{convergence19}
\lim_{n \to \infty} \int_0^t \int_0^{n} \Pi_{\Theta}(m) \mathcal{S}_n^R(m)  g_n(m, s) dm ds =
\int_0^t \int_0^{\infty} \Pi_{\Theta}(m) \mathcal{S}^R(m)  g(m, s) dm ds.
\end{align}
Combining \eqref{convergence15} and \eqref{convergence19}, we obtain
\begin{align}\label{convergence20}
\lim_{n \to \infty} \bigg\{ \int_0^t & \int_0^{n} \int_0^{n} \tilde{\Theta}_n(m, m^{\ast}) C^K_n(m, m^{\ast}) g_n(m, s) g_n(m^{\ast}, s) dm^{\ast} dm ds \nonumber\\
& +\int_0^t \int_0^{n} \Pi_{\Theta}(m) \mathcal{S}_n^R(m)  g_n(m, s) dm ds \bigg\} \nonumber\\
=& \int_0^t  \int_0^{\infty} \int_0^{\infty} \tilde{\Theta}(m, m^{\ast}) C^K(m, m^{\ast}) g(m, s) g(m^{\ast}, s) dm^{\ast} dm ds \nonumber\\
&+ \int_0^t \int_0^{\infty} \Pi_{\Theta}(m) \mathcal{S}^R(m)  g(m, s) dm ds.
\end{align}
It follows from \eqref{Convergenceweakly1} and \eqref{truncatedinidata} that
\begin{align}\label{convergence21}
\lim_{n \to \infty} \int_0^n \{ g_n(m, t) -g_n^{in}(m) \} \Theta(m) dm = \int_0^{\infty} \{ g(m, t) -g^{in}(m) \} \Theta(m) dm,
\end{align}
for every $\Theta \in L^{\infty} (\mathds{R}_{>0})$. The proof of Theorem \ref{Theorem1} follows from \eqref{convergence20} and \eqref{convergence21}.

%\begin{rmkk}
%A comment on Theorem \ref{Theorem1} is that when $\frac{\Gamma(m)}{m} \to 0$, as $m \to \infty$, this theorem is equivalent to \cite[Theorem 8.2.54]{Banasiak:2019}. This means we cover the result in \cite{Banasiak:2019}.
%\end{rmkk}

%%%%%%%%%%%%%%%%%%%%%%%%%%%%%%%%%%%%%%%%%
%%%%%%%%%%%%%%%%%%%%%%%%%%%%%%%%%%%%%%%%%
\section{Gelation in Smoluchowski coagulation model}
%%%%%%%%%%%%%%%%%%%%%%%%%%%%%%%%%%%%%%%%%
%%%%%%%%%%%%%%%%%%%%%%%%%%%%%%%%%%%%%%%%%
In this section, we discuss the gelation of existing solutions when the selection rate, $\mathcal{S}^R \equiv 0$ and the coagulation kernel satisfies \eqref{condicoag} and \eqref{coagkerGel}. Let us first state the following lemma.
\begin{lem}\label{Lemma51}
Let $\Theta \in L^{\infty}(\mathds{R}_{>0})$, $s<t$ and $s \in [0, \infty)$. Then, we have
\begin{align}\label{truncatedidentity51}
\int_0^{\infty} \Theta(m) \{ g(m, t) - & g(m, s) \} dm \nonumber\\
=&\frac{1}{2}\int_s^t \int_0^{\infty} \int_{0}^{\infty} \tilde{\Theta}(m, m^{\ast}) C^K(m, m^{\ast}) g(m, \tau) g(m^{\ast}, \tau) dm^{\ast} dm d\tau,
\end{align}
where $\tilde{\Theta}$ is defined in \eqref{Identity400}.
%\begin{align}
%\tilde{\Theta}(m, m^{\ast}):=\Theta(m+m^{\ast})-\Theta(m)-\Theta(m^{\ast}).
%\end{align}
\end{lem}
%%%%%%%%%%%%%%%%%%%%%%%%%%%%%%%%%%%%%%%
%%%%%%%%%%%%%%%%%%%%%%%%%%%%%%%%%%%%%%%
%%%%%%%%%%%%%%%%%%%%%%%%%%%%%%%%%%%%%%
%%%%%%%%%%%%%%%%%%%%%%%%%%%%%%%%%%%%%%
\begin{lem}\label{Lemma52}
For $t \in \mathds{R}_{>0}$, and $s \in [0, t)$, then
\begin{align}\label{Massinequalitylemma51}
 \mathcal{N}_1(t)  \le \mathcal{N}_1(s).
\end{align}

%In addition, let $\Omega: (0, \infty) \mapsto [0, \infty)$ be a nonnegative and non-increasing real valued function such that
%\begin{align}
%\Omega(m+m^{\ast}) \le \Omega(m)+\Omega(m^{\ast}), \ \ \ (m, m^{\ast}) \in (0, \infty)^2.
%\end{align}
% Then, if $g^{in}$ enjoys the additional integrability property
% \begin{align}\label{Conditioninitialdata}
%A:=\int_0^{\infty} \Omega(m) g^{in}(m) dm < \infty,
% \end{align}
% so does $g(., t)$ and $t >s \ge 0$,
% \begin{align}
% \int_0^{\infty} \Omega(m) g(m, t) dm \le \int_0^{\infty} \Omega(m) g(m, s) dm.
% \end{align}
\end{lem}
%%%%%%%%%%%%%%%%%%%%%%%%%%%%%%%%%%%%%%%%%%%%
%%%%%%%%%%%%%%%%%%%%%%%%%%%%%%%%%%%%%%%%%%%%
\begin{proof}
Let $\lambda^{\dag} \in \mathds{R}_{>0}$ and take $\Theta(m):= \lambda^{\dag} \wedge m $, then the corresponding $\tilde{ \Theta}$ is
 \[
\tilde{ \Theta}(m, m^{\ast}):=\begin{cases}
0,\ & \text{if}\ m+m^{\ast} < \lambda^{\dag}, \ m < \lambda^{\dag},\ m^{\ast} < \lambda^{\dag}, \\
\lambda^{\dag}-(m+m^{\ast}),\ &  \text{if}\ m+m^{\ast} \ge \lambda^{\dag},\ m < \lambda,\ m^{\ast} < \lambda^{\dag},\\
-m, \ &  \text{if}\ m+m^{\ast} \ge \lambda^{\dag},\ m < \lambda^{\dag},\ m^{\ast} \ge \lambda^{\dag},\\
- m^{\ast}, \  &  \text{if}\ m+m^{\ast} \ge \lambda^{\dag},\ m \ge \lambda^{\dag},\ m^{\ast} < \lambda^{\dag},\\
-\lambda^{\dag},\ &  \text{if}\ m+m^{\ast} \ge \lambda^{\dag},\ m \ge \lambda^{\dag},\ m^{\ast} \ge  \lambda^{\dag}.
\end{cases}
\]
From above it is clear that $\tilde{ \Theta} \le 0$. Using this $\tilde{ \Theta } \le 0$ into Lemma \ref{truncatedidentity51}, we obtain
\begin{align*}
\int_0^{\infty} ( \lambda^{\dag} \wedge m ) g(m, t) dm \le \int_0^{\infty} ( \lambda^{\dag} \wedge m ) g(m, s) dm.
\end{align*}
As $\lambda^{\dag} \to \infty $, we get \eqref{Massinequalitylemma51}.

%Next, the function $\Omega$ being as in Lemma \ref{Lemma52}, we define
%\begin{align*}
%\Omega_{\epsilon }(m) =\min (\Omega(\epsilon),  \Omega(m)), \ \ \ m \in (0, \infty),\ \text{for}\ \epsilon \in (0, 1).
%\end{align*}
%Since $\Omega$ is non-negative, hence by definition of $\Omega_{\epsilon}$ is also non-negative and non-increasing bounded function in $(0, \infty)$, and
%\begin{align*}
%\Omega_{\epsilon }(m+m^{\ast}) \le  \Omega_{\epsilon }(m) +\Omega_{\epsilon }(m^{\ast}).
%\end{align*}
%In addition, for each $m \in (0, \infty)$ there holds
%\begin{align*}
%\lim_{\epsilon \to 0 }  \Omega_{\epsilon }(m) =\Omega(m).
%\end{align*}
%Now, we take $\Theta = \Omega_{\epsilon } $ in Lemma \ref{Lemma51}, and obtain,
%\begin{align}\label{Mass Inequality2}
%\int_0^{\infty} \Omega_{\epsilon }(m) g(m, t) dm  \le \int_0^{\infty} \Omega_{\epsilon}(m) g(m, s) dm.
%\end{align}
%We first take $s=0$ and $t=s$ in \eqref{Mass Inequality2} and $\epsilon \to 0$. The monotone convergence theorem and \eqref{Conditioninitialdata} entail that
%\begin{align}\label{Mass Inequality3}
%\lim_{\epsilon \to 0} \int_0^{\infty} \Omega_{\epsilon}(m) g(m, s) dm \le \int_0^{\infty} \Omega(m) g^{in}(m) dm =A.
%\end{align}
%Again take $\epsilon \to 0$ in \eqref{Mass Inequality2} and using \eqref{Mass Inequality3}, we have
%\begin{align}\label{Mass Inequality4}
%\int_0^{\infty} \Omega(m) g(m, t) dm  \le \int_0^{\infty} \Omega(m) g(m, s) dm.
%\end{align}
%This proves Lemma \ref{Lemma52}.
\end{proof}
Now, this is the right time to complete the proof of Theorem \ref{Theorem2}.
\begin{proof} of Theorem \ref{Theorem2}, let $s \in [0, \infty)$ and $t \in (s, \infty)$. We take $\Theta \equiv 1$ into \eqref{truncatedidentity51} to
obtain
\begin{align}\label{Equal51}
\int_0^{\infty}   g(m, t)  dm +\frac{1}{2}\int_s^t \int_0^{\infty} \int_{0}^{\infty}  C^K(m, m^{\ast}) g(m, \tau)
g(m^{\ast}, \tau)dm^{\ast} dm d\tau
=\int_0^{\infty}  g(m, s) dm.
\end{align}
Using \eqref{condicoag} and \eqref{coagkerGel} into \eqref{Equal51}, we obtain
\begin{align}\label{Equal52}
 \frac{1}{2}  k_1 \lambda^2 \int_s^t \int_0^{\infty} \int_{0}^{\infty}  m m^{\ast} g(m, \tau)
g(m^{\ast}, \tau)dm^{\ast} dm d\tau \le \int_0^{\infty}  g(m, s) dm.
\end{align}
In particular, put $s=0$ into \eqref{Equal52}, we get
\begin{align}\label{Equal53}
 \int_0^t \mathcal{N}_1^2(\tau) d\tau \le  \frac{2}{ k_1 \lambda^2}  \int_0^{\infty}  g^{in}(m) dm.
\end{align}
It follows from \eqref{Massinequalitylemma51} and \eqref{Equal53} that
\begin{align*}
 t \mathcal{N}_1^2(t)  \le  \frac{2}{ k_1 \lambda^2} \mathcal{N}^{in}_0.
\end{align*}
This proves \eqref{GelationCoa1}. In order to prove \eqref{GelationCoa3}, consider $\Theta(m):= (m +\zeta)^{-p}$, for $\zeta \in (0, 1)$ and $p \in \mathds{R}_{>0}$. Then the corresponding $\tilde{\Theta}$ is
\begin{align}\label{Theta41}
\tilde{\Theta}(m, m^{\ast})= (m +m^{\ast}+\zeta)^{-p}-(m +\zeta)^{-p}-(m^{\ast} +\zeta)^{-p} \le 0.
\end{align}
Next, using \eqref{Theta41} and set $s=0$ into \eqref{Lemma51}, we have
\begin{align*}
\int_0^{\infty} (m+\zeta)^{-p} g(m, t) dm \le \int_0^{\infty} (m+\zeta)^{-p} g^{in}(m) dm.
\end{align*}
As $\zeta \to 0$, it can be inferred from Fatou's lemma that
\begin{align}\label{Equal54}
\int_0^{\infty} m^{-p} g(m, t) dm \le \mathcal{I}_p = \int_0^{\infty} m^{-p} g^{in}(m) dm.
\end{align}
By using the H\"{o}lder's inequality, we deduce that
\begin{align}\label{Equal55}
\int_0^{\infty} g(m, t) dt = & \int_0^{\infty} m^{\frac{p}{p+1}} g(m, t)^{\frac{p}{p+1}} m^{-\frac{p}{p+1}} g(m, t)^{\frac{1}{p+1}} dm \nonumber\\
\le & \bigg[\int_0^{\infty} \bigg(m^{\frac{p}{p+1}} g(m, t)^{\frac{p}{p+1}} \bigg)^{\frac{p+1}{p}} dm \bigg]^{\frac{p}{p+1}}
\bigg[ \int_0^{\infty}  \bigg( m^{-\frac{p}{p+1}} g(m, t)^{\frac{1}{p+1}} \bigg)^{p+1} dm \bigg]^{\frac{1}{p+1}} \nonumber\\
\le & {\mathcal{N}_1(t)}^{\frac{p}{p+1}} \mathcal{I}_p^{\frac{1}{p+1}}.
\end{align}
From \eqref{Equal52} and \eqref{Equal55}, we have
\begin{align*}
 \int_s^t \mathcal{N}_1^2(\tau) d\tau \le   \frac{2}{  k_1 \lambda^2} {\mathcal{N}_1(t)}^{\frac{p}{p+1}} \mathcal{I}_p^{\frac{1}{p+1}}.
\end{align*}
Above inequality is true for every $t>s$, thus, we have
\begin{align}\label{Equal56}
 \int_s^{\infty} \mathcal{N}_1^2(\tau) d\tau \le c  {\mathcal{N}_1(s)}^{\frac{p}{p+1}},
\end{align}
where $c=  \frac{2}{  k_1 \lambda^2} \mathcal{I}_p^{\frac{1}{p+1}}$. Taking derivative of \eqref{Equal56} with respect to $s$ and applying Leibnitz's rule, we obtain
\begin{align*}
 - \mathcal{N}_1^2(s)  \le  \frac{cp}{p+1}  {\mathcal{N}_1(s)}^{\frac{-1}{p+1}} \frac{d \mathcal{N}_1(s)}{ds}.
\end{align*}
This implies that
\begin{align}\label{Equal57}
 - \frac{(p+1)}{cp} ds  \le    {\mathcal{N}_1(s)}^{\frac{-(2p+3)}{p+1}} d \mathcal{N}_1(s).
\end{align}
Integrating \eqref{Equal57} from $0$ to $t$ and simplifying it, we get
\begin{align*}
  \mathcal{N}_1(t)^{\frac{-(p+2)}{(p+1)} }\le \frac{(p+2)t}{cp} + { \mathcal{N}_1^{in}}^{\frac{-(p+2)}{(p+1)}}.
\end{align*}
This implies that
%\begin{align}\label{Equal58}
%  \mathcal{N}_1(t) \le { \mathcal{N}_1^{in}} \bigg[ 1+ { \mathcal{N}_1^{in}}^{\frac{(p+2)}{(p+1)}} \frac{(p+2)tk_1 \lambda^2 \mathcal{I}_p^{\frac{-1}{(p+1)}}}{2p} \bigg]^{\frac{-(p+1)}{(p+2)}}.
%\end{align}
\begin{align*}
  \mathcal{N}_1(t) \le { \mathcal{N}_1^{in}} \bigg[ 1+  T^{\ddag}t  \bigg]^{\frac{-(p+1)}{(p+2)}},
\end{align*}
where $T^{\dag}=  \frac{(p+2)}{2p}  { \mathcal{N}_1^{in}}^{\frac{(p+2)}{(p+1)}} k_1 \lambda^2  \mathcal{I}_p^{\frac{-1}{(p+1)}}$. One can see that the gelation starts for $t= -1/T^{\dag}$.

Finally, to prove \eqref{GelationCoa4}, we take $\Theta(m)= \chi_{(0, \delta)}(m)$ in Lemma \ref{Lemma51}, then the corresponding $\tilde{ \Theta}$ is
 \[
\tilde{ \Theta}(m, m^{\ast}):=\begin{cases}
-1 ,\ & \text{if}\ m+m^{\ast} < \delta, \ m < \delta,\ m^{\ast} < \delta, \\
-2,\ &  \text{if}\ m+m^{\ast} \ge \delta,\ m < \delta,\ m^{\ast} < \delta,\\
-1, \ &  \text{if}\ m+m^{\ast} \ge \delta,\ m < \delta,\ m^{\ast} \ge \delta,\\
-1, \  &  \text{if}\ m+m^{\ast} \ge \delta,\ m \ge \delta,\ m^{\ast} < \delta,\\
0,\ &  \text{if}\ m+m^{\ast} \ge \delta,\ m \ge \delta,\ m^{\ast} \ge  \delta.
\end{cases}
\]
and for $s=0$, we have
\begin{align*}
\int_0^{\delta} g(m, t)  dm \le \int_0^{\delta}    g^{in}(m)  dm.
\end{align*}
It follows from the non-negativity of $g$ and $g^{in}\equiv 0 $ on $(0, \delta)$ that
\begin{align*}
\int_0^{\delta} g(m, t)  dm =0,
\end{align*}
for every $t \in \mathds{R}_{>0}$. Thus,
\begin{align}\label{Equal58}
 g(m, t) =0,\ \text{a.e. in}\ (0, \delta)\ \text{and}\ t \in \mathds{R}_{>0}.
\end{align}
From \eqref{Equal53}, we have
\begin{align*}
\int_s^t   \mathcal{N}_1^2(\tau) d\tau \le \frac{2}{ k_1 \lambda^2}  \int_{\delta}^{\infty}  g(m, s) dm \le \frac{2}{ k_1 \delta \lambda^2} \mathcal{N}_1(s).
\end{align*}
This is true for every $t>s$. Hence, we obtain
\begin{align}\label{Equal58}
\int_s^{\infty}   \mathcal{N}_1^2(\tau) d\tau \le \frac{2}{ k_1 \lambda^2}  \int_{\delta}^{\infty}  g(m, s) dm \le \frac{2}{ k_1 \delta \lambda^2} \mathcal{N}_1(s).
\end{align}
Again, differentiating by using Leibnitz's rule, we get
\begin{align*}
- \frac{ k_1 \delta \lambda^2}{2} ds  \le  \frac{d\mathcal{N}_1(s)}{\mathcal{N}_1^2(s)}.
\end{align*}
On integration yields with respect to $s$ from $0$ to $t$
\begin{align*}
\frac{1}{\mathcal{N}_1^2(t)}  \le \frac{1}{{\mathcal{N}_1^{in}}^2} + \frac{ k_1 \delta \lambda^2}{2} t.
\end{align*}
This implies that
\begin{align*}
\mathcal{N}_1(t)  \le  \mathcal{N}_1^{in}\sqrt{2} [ 2+ k_1 \delta \lambda^2 t {\mathcal{N}_1^{in}}^2  ]^{\frac{-1}{2}}.
\end{align*}
This completes the proof of theorem \ref{Theorem2}.

\end{proof}

In the next section, we discuss the gelation in continuous coagulation and multiple fragmentation model.
%%%%%%%%%%%%%%%%%%%%%%%%%%%%%%%%%%%%%
%%%%%%%%%%%%%%%%%%%%%%%%%%%%%%%%%%%%%
\section{Gelation in continuous coagulation and multiple fragmentation model}
%%%%%%%%%%%%%%%%%%%%%%%%%%%%%%%%%%%%%
%%%%%%%%%%%%%%%%%%%%%%%%%%%%%%%%%%%%%
\begin{lem}\label{Identity Multifragment}
Let $\Theta \in L^{\infty}(\mathds{R}_{>0})$. For $t \in \mathds{R}_{>0}$ and $s \in [0, t)$, we have
\begin{align*}
\int_0^{\infty} \{ g(m, t) & - g(m, s) \} \Theta(m)dm \nonumber\\
=&\frac{1}{2}\int_s^t \int_0^{\infty} \int_{0}^{\infty} \tilde{\Theta}(m, m^{\ast}) C^K(m, m^{\ast}) g(m, \tau) g(m^{\ast}, \tau)dm^{\ast} dm d\tau \nonumber\\
&+\int_s^t \int_0^{\infty} \Pi_{\Theta}(m)S^R(m) g(m, \tau)dm d\tau,
\end{align*}
where $\tilde{\Theta}$ and $\Pi_{\Theta}$ are defined in \eqref{Identity400} and \eqref{1Identity2}, respectively.
\end{lem}
%%%%%%%%%%%%%%%%%%%%%%%%%%%%%%%%%%%%
%%%%%%%%%%%%%%%%%%%%%%%%%%%%%%%%%%%%

%%%%%%%%%%%%%%%%%%%%%%%%%%%%%%%%%%%%
%%%%%%%%%%%%%%%%%%%%%%%%%%%%%%%%%%%%
\begin{lem}\label{LemmaestimateCMFE}
 For $t \in \mathds{R}_{>0}$ and $s \in [0, t)$, then we obtain
 \begin{align}\label{LemmaestimateCMFEeq2}
 \mathcal{N}_1(t)  \le   \mathcal{N}_1(s),
\end{align}
%\begin{align}\label{LemmaestimateCMFEeq3}
% \mathcal{N}_0(t)  \le   \mathcal{N}_0^{in}+\mathcal{N}_1^{in}T e^{k_3(\eta-1) \varphi(0)t},
%\end{align}
and
\begin{align}\label{LemmaestimateCMFEeq1}
\int_0^t | \mathcal{N}_1(\tau)|^2 d\tau \le   \frac{2}{\lambda^2} \mathcal{Q} + \frac{2}{\lambda^2} k_3 (\eta -1) \varphi(0) \int_0^t \mathcal{N}_1(\tau)d\tau.
\end{align}
\end{lem}
%%%%%%%%%%%%%%%%%%%%%%%%%%%%%%%%%%%%%%%%%%
%%%%%%%%%%%%%%%%%%%%%%%%%%%%%%%%%%%%%%%%%%
\begin{proof}
Set $\Theta(m)= \min\{m, \lambda^{\dag} \}$, then the corresponding  $\tilde{ \Theta}$ is
 \[
\tilde{ \Theta}(m, m^{\ast}):=\begin{cases}
0,\ & \text{if}\ m+m^{\ast} < \lambda^{\dag}, \ m < \lambda^{\dag},\ m^{\ast} < \lambda^{\dag}, \\
\lambda^{\dag}-(m+m^{\ast}),\ &  \text{if}\ m+m^{\ast} \ge \lambda^{\dag},\ m < \lambda,\ m^{\ast} < \lambda^{\dag},\\
-m, \ &  \text{if}\ m+m^{\ast} \ge \lambda^{\dag},\ m < \lambda^{\dag},\ m^{\ast} \ge \lambda^{\dag},\\
- m^{\ast}, \  &  \text{if}\ m+m^{\ast} \ge \lambda^{\dag},\ m \ge \lambda^{\dag},\ m^{\ast} < \lambda^{\dag},\\
-\lambda^{\dag},\ &  \text{if}\ m+m^{\ast} \ge \lambda^{\dag},\ m > \lambda^{\dag},\ m^{\ast} > \lambda^{\dag}.
\end{cases}
\]
 It is clear that $\tilde{ \Theta}(m, m^{\ast}) \le 0$ and $\Pi_{\Theta}(m) \le \lambda^{\dag} (\eta -1)$. It follows from Lemma \ref{Identity Multifragment} that
\begin{align}\label{Identity505}
\int_0^{\infty} \min\{m, \lambda^{\dag} \} & \{ g(m, t)  - g(m, s) \} dm  \nonumber\\
\le  & k_3 \lambda^{\dag} (\eta -1)  \varphi(\lambda^{\dag}) \int_s^t \int_{\lambda^{\dag}}^{\infty} m g(m, \tau)dm d\tau.
\end{align}
This implies that
\begin{align}\label{Identity506}
\int_0^{\lambda^{\dag}}  m \{ g(m, t)  - g(m, s) \} & dm + \int_{\lambda^{\dag}}^{\infty} \lambda^{\dag}   \{ g(m, t)  - g(m, s) \} dm \nonumber\\
\le & k_3 \lambda^{\dag} (\eta -1)  \varphi(\lambda^{\dag}) \int_s^t \int_{\lambda^{\dag}}^{\infty} m g(m, \tau)dm d\tau.
\end{align}
Since $g \in L^{\infty}(0, t; L^1_{-2\sigma, 1})$ then as $\lambda^{\dag} \to \infty$ and the Lebesgue's dominated convergence theorem, we obtain \eqref{LemmaestimateCMFEeq2}.
%Next, we take $\Theta \equiv 1$ and put $s=0$ in Lemma \ref{Identity Multifragment} and then using \eqref{TNP}, we have
%\begin{align}\label{Identity501}
%\int_0^{\infty} \{ g(m, t)  - g^{in}(m) \} dm
%+&\frac{1}{2}\int_0^t \int_0^{\infty} \int_{0}^{\infty}  C^K(m, m^{\ast}) g(m, \tau) g(m^{\ast}, \tau)dm^{\ast} dm d\tau \nonumber\\
%=& (\eta-1) \int_0^t  \int_0^{\infty}  S^R(m) g(m, \tau)dm d\tau.
%\end{align}
%Using \eqref{Selection Rate} and \eqref{LemmaestimateCMFEeq2} into \eqref{Identity501}, we obtain
%\begin{align}\label{Identity5011}
%  \mathcal{N}_0 (t) -   \mathcal{N}_0^{in}   \le & k_3(\eta-1) \varphi(0) \int_0^t \int_0^{\infty}  m^{1+\gamma} g(m, \tau)dm d\tau\nonumber\\
%\le & k_3(\eta-1) \varphi(0) \int_0^t  \{ \mathcal{N}_0 (\tau)+  \mathcal{N}_1^{in} \} d\tau.
%\end{align}
%The Gronwall's inequality gives
%\begin{align}\label{Identity5012}
%  \mathcal{N}_0 (t)    \le & \mathcal{N}_0^{in} + k_3(\eta-1) \varphi(0)t  +  \bigg(   \mathcal{N}_0^{in}+ \mathcal{N}_1^{in} \bigg) e^{k_3(\eta-1) \varphi(0)t}.
%\end{align}
Again, using \eqref{coagkerGel}, \eqref{SelectionRate1} and \eqref{Initialdatafinite} into \eqref{Identity506} and set $s=0$, we get
\begin{align}\label{Identity502}
\frac{\lambda^2}{2} \int_0^t \int_0^{\infty} \int_{0}^{\infty} &  m m^{\ast} g(m, \tau) g(m^{\ast}, \tau)dm^{\ast} dm d\tau \nonumber\\
\le & \int_0^{\infty} g^{in}(m) dm +k_3 (\eta -1) \int_0^t \int_0^{\infty} \varphi(m) m  g(m, \tau)dm d\tau \nonumber\\
\le & \mathcal{Q} + k_3 (\eta -1) \varphi(0) \int_0^t \int_0^{\infty}  m g(m, \tau)dm d\tau.
\end{align}
Next, applying \eqref{condicoag} to \eqref{Identity502}, we have
\begin{align*}
\int_0^t \mathcal{N}_1^2(\tau) d\tau \le & \frac{2}{\lambda^2} \mathcal{Q} + \frac{2}{\lambda^2} k_3 (\eta -1) \varphi(0)
  \int_0^t \mathcal{N}_1(\tau)d\tau  \nonumber\\
  \le &  \frac{2}{\lambda^2} \mathcal{Q} + \frac{2}{\lambda^2} k_3 (\eta -1) \varphi(0) \int_0^t \mathcal{N}_1(\tau)d\tau.
\end{align*}
This proves \eqref{LemmaestimateCMFEeq1}.
\end{proof}
%%%%%%%%%%%%%%%%%%%%%%%%%%%%%%%%%%%%%%%%%
%%%%%%%%%%%%%%%%%%%%%%%%%%%%%%%%%%%%%%%%%

Now it is the right time to complete the proof of the Theorem \ref{Theorem3}.
%%%%%%%%%%%%%%%%%%%%%%%%%%%%%%%%%%%%%%%%
%%%%%%%%%%%%%%%%%%%%%%%%%%%%%%%%%%%%%%%%
\begin{proof} of Theorem \ref{Theorem3}: For $t \in \mathds{R}_{>0}$, we put
\begin{align*}
\mathcal{I}(t) := \int_0^t \mathcal{N}_1(\tau ) d\tau.
\end{align*}
On the one hand, it follows from the Jensen's inequality that
\begin{align*}
\mathcal{I}^2(t) =\bigg( \int_0^t \mathcal{N}_1(\tau ) d\tau \bigg)^2 \le \bigg( \int_0^t 1^2 d\tau \bigg) \bigg( \int_0^t \mathcal{N}_1^2(\tau ) d\tau \bigg) \le t \int_0^t \mathcal{N}_1^2 (\tau )d\tau.
\end{align*}
Substituting \eqref{LemmaestimateCMFEeq1} into above inequality, we obtain
\begin{align}\label{5IntegrabilityInequality1}
\mathcal{I}^2(t)  \le  & \frac{2t}{\lambda^2} \mathcal{Q} + \frac{2t}{\lambda^2} k_3 (\eta -1) \varphi(0) \int_0^t \mathcal{N}_1(\tau)d\tau \nonumber\\
 =  & \frac{2t}{\lambda^2} \mathcal{Q} + \frac{2t}{\lambda^2} k_3 (\eta -1) \varphi(0) \mathcal{I}(t).
\end{align}
After solving above quadratic inequality \eqref{5IntegrabilityInequality1}, we get
\begin{align}\label{5IntegrabilityInequality2}
\mathcal{I}(t)  \le \frac{k_3}{\lambda^2} t (\eta -1) \varphi(0) + \bigg[ \frac{k_3^2}{\lambda^4} t^2(\eta-1)^2 \varphi(0)^2+\frac{2\mathcal{Q}t} {\lambda^2}   \bigg]^{1/2}
\end{align}
From \eqref{LemmaestimateCMFEeq2} and \eqref{5IntegrabilityInequality2}, we obtain
\begin{align*}
\mathcal{N}_1(t)  \le \frac{k_3}{\lambda^2}  (\eta -1) \varphi(0) + \bigg[ \frac{k_3^2}{\lambda^4} (\eta-1)^2 \varphi(0)^2+\frac{2\mathcal{Q}} {t\lambda^2} \bigg]^{1/2}
\end{align*}
Finally, as $t \to \infty$, we thus have
\begin{align*}
\lim_{t \to \infty}\mathcal{N}_1(t)  \le \frac{2k_3}{\lambda^2}  (\eta -1) \varphi(0).
\end{align*}
This completes the proof.
\end{proof}
%\section{Uniqueness of weak solutions}
%%%%%%%%%%%%%%%%%%%%%%%%%%%%%%%%%%%%%%%
%%%%%%%%%%%%%%%%%%%%%%%%%%%%%%%%%%%%%%%
%\section*{Acknowledgments}

%%%%%%%%%%%%%%%%%%%%%%%%%%%%%%%%%%%%%%%%%%%%%%%%%%%%%%%%%%%%%%%%%%%
%%%%%%%%%%%%%%%%%%%%%%%%%%%%%%%%%%%%%%%%%%%%%%%%%%%%%%%%%%%%%%%%%%%
%The author author wants to thank Tata Institute of Fundamental Research Centre for Applicable Mathematics for its funding support
%during my post-doctoral sudies. %The authors are truly thankful to the anonymous reviewers for providing insightful
%%comments that improved the readability of the article.

%%%%%%%%%%%%%%%%%%%%%%%%%%%%%%%%%%%%%%%%
%%%%%%%%%%%%%%%%%%%%%%%%%%%%%%%%%%%%%%%%
\bibliographystyle{plain}

%%%%%%%%%%%%%%%%%%%%%%%%%%%%%%%%%
%%%%%%%%%%%%%%%%%%%%%%%%%%%%%%%%%

\end{document}